\documentclass[11pt, a4paper]{amsart}
\usepackage{a4wide}
\usepackage{graphicx,enumitem}
\usepackage{subfigure}
\usepackage{units}
\usepackage{color}
\usepackage{url}
\usepackage{cancel}

\usepackage{ulem}
\usepackage{latexsym}

\usepackage[pdftex,
            pdftitle={Numerical analysis of lognormal diffusions on the sphere},
            pdfauthor={L. Herrmann and A. Lang and Ch. Schwab},
            bookmarksopen,
            colorlinks,
            linkcolor=black,
            urlcolor=black,
	    citecolor=black
]{hyperref}

\usepackage{float}
\restylefloat{figure}

\usepackage{dsfont}

\allowdisplaybreaks
\usepackage{mathtools}

\newcommand{\IP}{\mathbb{P}}
\newcommand{\R}{\mathbb{R}}
\newcommand{\C}{\mathbb{C}}
\newcommand{\N}{\mathbb{N}}

\newcommand{\IS}{\mathbb{S}}

\newcommand{\IA}{\mathbb{A}}

\newcommand{\gb}{\beta}

\renewcommand{\gg}{\gamma}
\newcommand{\gk}{\kappa}

\newcommand{\go}{\omega}
\newcommand{\gs}{\sigma}

\newcommand{\vp}{\varphi}
\newcommand{\vt}{\vartheta}

\newcommand{\gO}{\Omega}

\newcommand{\cA}{\mathcal{A}}
\newcommand{\cB}{\mathcal{B}}

\newcommand{\cG}{\mathcal{G}}
\newcommand{\cH}{\mathcal{H}}
\newcommand{\cI}{\mathcal{I}}

\newcommand{\cN}{\mathcal{N}}

\newcommand{\cT}{\mathcal{T}}

\newcommand{\cW}{\mathcal{W}}

\newcommand{\cY}{\mathcal{Y}}

\newcommand{\Op}{\operatorname{O}}

\newcommand{\dd}{\mathrm{d}}

\DeclareMathOperator{\E}{\mathbb{E}} 

\DeclareMathOperator{\Span}{span}
\DeclareMathOperator{\supp}{supp}

\DeclareMathOperator{\Id}{Id}

\let\Re\relax
\DeclareMathOperator{\Re}{Re}
\let\Im\relax
\DeclareMathOperator{\Im}{Im}

\newcommand{\KL}{Karhunen--Lo\`eve }

\newtheorem{lemma}{Lemma}[section]
\newtheorem{proposition}[lemma]{Proposition}

\newtheorem{theorem}[lemma]{Theorem}

\newtheorem{corollary}[lemma]{Corollary}
\theoremstyle{remark}
\newtheorem{remark}[lemma]{Remark}

\theoremstyle{definition}

\begin{document}
\title[Lognormal diffusions on the sphere]{
Numerical analysis of lognormal diffusions on the sphere
}

\author[L.~Herrmann]{Lukas Herrmann} 
\address[Lukas Herrmann]{\newline Seminar f\"ur Angewandte Mathematik
\newline ETH Z\"urich, \newline R\"amistrasse 101, CH--8092 Z\"urich, Switzerland.
} \email[]{lukas.herrmann@sam.math.ethz.ch}

\author[A.~Lang]{Annika Lang} \address[Annika Lang]{\newline Department of Mathematical Sciences
\newline Chalmers University of Technology \& University of Gothenburg, 
\newline SE--412 96 G\"oteborg, Sweden. } \email[]{annika.lang@chalmers.se}

\author[Ch.~Schwab]{Christoph Schwab} \address[Christoph Schwab]{\newline Seminar f\"ur Angewandte Mathematik 
\newline ETH Z\"urich, \newline R\"amistrasse 101, CH--8092 Z\"urich, Switzerland.} 
\email[]{schwab@math.ethz.ch}

\thanks{%
\emph{Acknowledgment}.
The work was supported in part by the European Research Council under 
ERC~AdG~247277, the Swiss National Science Foundation under SNF~200021\_159940/1,
the Knut and Alice Wallenberg foundation,
and the Swedish Research Council under Reg.~No.~621-2014-3995.
We thank three anonymous referees for comments that improve the presentation.
}

\date{\today}
\subjclass{60G60, 60G15, 60G17, 33C55, 41A25, 60H15, 60H35, 65C30, 65N30}
\keywords{
Isotropic Gaussian random fields, lognormal random fields, \KL expansion, 
spherical harmonic functions, stochastic partial differential equations, 
random partial differential equations, 
regularity of random fields, Finite Element Methods, Spectral Galerkin Methods,
multilevel Monte Carlo methods%
}

\begin{abstract}
Numerical solutions of
stationary diffusion equations on the unit sphere with isotropic 
lognormal diffusion coefficients are considered.
H\"older regularity in $L^p$ sense for isotropic Gaussian random fields is obtained and 
related to the regularity of the driving lognormal coefficients. This yields regularity in $L^p$ sense 
of the solution to the diffusion problem in Sobolev spaces.
Convergence rate estimates of 
multilevel Monte Carlo Finite and Spectral Element discretizations 
of these problems are then deduced. 
Specifically, a convergence analysis 
is provided with convergence rate estimates in terms of
the number of Monte Carlo samples of the solution to the considered diffusion equation
and 
in terms of the total number of degrees of freedom of the spatial discretization,
and with bounds for the total work 
required by the algorithm in the case of Finite Element discretizations.
The obtained convergence rates
are solely in terms of the decay of the angular power spectrum
of the (logarithm) of the diffusion coefficient.
Numerical examples confirm the presented theory.
\end{abstract}

\maketitle
\tableofcontents

%
\section{Introduction}
In the present paper, we are concerned with the existence,
regularity, and approximation of solutions of elliptic partial
differential equations (PDEs for short) with stochastic coefficients
on the unit sphere~$\IS^2$.
In particular, we are interested in PDEs with
\emph{isotropic lognormal random field} coefficients $a$, i.e., $T = \log a$ 
is an isotropic Gaussian random field (iGRF for short) on $\IS^2$.
For a given smooth, deterministic source term $f$, and for a positive 
random field $a$ taking values in $C^0(\IS^2)$,
we consider the stochastic elliptic problem
\begin{equation}\label{eq:ellSPDE}
-\nabla_{\IS^2} \cdot (a \nabla_{\IS^2} u ) = f \quad \text{on}\; \IS^2.
\end{equation}

Since $\partial \IS^2 = \emptyset$ (as boundary of a manifold),
no boundary conditions 
are required for the well-posedness of~\eqref{eq:ellSPDE}. 
The regularity and integrability of solutions in terms of the random field~$a$ as well as
error and convergence rate analysis of 
Finite Element and Spectral Galerkin discretizations on~$\IS^2$ 
combined with multilevel Monte Carlo (MLMC for short) sampling 
are the purpose of the present paper. 

While the combined Finite Element MLMC discretization of PDEs with random input
data has received considerable attention in recent years (see, for example, 
\cite{ChScTe_SINUM13,gknsss15} and the survey \cite{GilesMLMC2015} 
originating from Heinrich \cite{HeinrichMLMC}),
the invariance properties of the particular geometry $\IS^2$ entail several specific
consequences in the numerical analysis which allow more precise convergence results.
Specifically, as we showed in~\cite{LS15, Herrmann13}, 
the geometric setting of~$\IS^2$ allows for an essentially sharp characterization
of H\"older regularity exponents of realizations of $a$ in terms of the angular power spectrum 
of the \KL expansion of the Gaussian random field $T = \log a$. 
Furthermore, $\partial \IS^2 = \emptyset$ implies the absence of corner 
singularities. 
We are therefore able to obtain
elliptic regularity estimates in Sobolev scales, cp.\ ~\cite{hackbusch2010Elldifequ}, 
as well as Schauder estimates of
classical elliptic regularity theory as presented for example in~\cite{GilTrud2ndEd} 
and elaborated in detail for the presently considered 
PDE~\eqref{eq:ellSPDE} in~\cite{Herrmann13}.
Based on these we derive
explicit convergence rate bounds of discretizations of~\eqref{eq:ellSPDE}.
Particularly,  we obtain convergence rates with respect to the mesh width of 
Finite Element discretizations and to the spectral degree of Spectral Galerkin discretizations on~$\IS^2$ solely in terms of the decay of the angular power spectrum of the Gaussian random field $T = \log a$.
These convergence rates are, in the Finite Element case, 
bounded by the polynomial degree of the basis functions.
We confine our error analysis to sufficiently smooth source terms~$f$ in~\eqref{eq:ellSPDE},
which yields that the lack of smoothness of solutions is caused by the roughness of the lognormal random coefficients~$a$.

Throughout the paper, we employ standard notation. 
We denote in particular by $H^s(\IS^2)$ Sobolev spaces of square integrable 
functions of (not necessarily integer) order $s$  on~$\IS^2$.
By $\nabla_{\IS^2}$, $\nabla_{\IS^2}\cdot$, and by
$\Delta_{\IS^2} = \nabla_{\IS^2}\cdot \nabla_{\IS^2}$
we denote the spherical gradient, the spherical divergence, and the
Laplace--Beltrami operator on $\IS^2$, respectively.

The outline of the paper is as follows:
In Section~\ref{sec:iGRF} we recapitulate basic properties of
iGRFs from~\cite{MP11,Charrier_SINUM12}.
We introduce standard notation and classical results 
from the differential geometry of surfaces as required 
in the ensuing developments.
We also review results on the H\"older regularity
of realizations of the random field from our earlier work~\cite{LS15}, 
and relate the H\"older exponent 
to the angular power spectrum. 
We develop H\"older regularity here in the $L^p$ sense.
In Section~\ref{sec:ExUniqRegSol}
we review and establish basic results
on existence, uniqueness, integrability, and regularity of solutions
to the stochastic partial differential equation (SPDE for short)~\eqref{eq:ellSPDE}.
In Section~\ref{sec:Discr} we present isoparametric
Finite Element (FE for short) discretizations
of  the SPDE~\eqref{eq:ellSPDE} on~$\IS^2$ and
establish a priori estimates on their convergence.
Particular attention is given to the dependence of the convergence rate
on the H\"older regularity of the random field~$a$.
In Section~\ref{sec:Discr}, 
we prove convergence rate estimates for two families of discretizations
of~\eqref{eq:ellSPDE}. Section~\ref{sec:FEMS2} is devoted to the analysis
of Finite Element discretizations, while Section~\ref{sec:SpecElS2} 
to the convergence analysis of Spectral Galerkin discretizations.
In Section~\ref{sec:MLMC} we address the
convergence of multilevel Monte Carlo methods
for either variant of the Galerkin discretizations.
Numerical examples that confirm the presented theory are presented in Section~\ref{sec:num_exp}.
Finally, some lengthy proofs are given in the appendix.

\section{Isotropic Gaussian and lognormal random fields}\label{sec:iGRF}

In this section we introduce isotropic Gaussian random fields on the unit sphere~$\IS^2$ and their 
properties. We focus in particular on \KL expansions of these random fields and their regularity in terms of H\"older continuity and $L^p$ integrability. Furthermore, similar results are presented for spectral approximations as well as the corresponding lognormal random fields. The section is based on results from~\cite{MP11} and \cite{LS15} and follows closely the master's thesis~\cite{Herrmann13} of one of the authors.

Let the unit sphere~$\IS^2$ in~$\R^3$ be given by
\begin{equation*}
 \IS^2 := \{ x \in \R^3, \|x\|_{\R^3} = 1\},
\end{equation*}
where $\| \cdot \|_{\R^3}$ denotes the Euclidean norm on~$\R^3$. 
Consider the compact metric space $(\IS^2,d)$ with geodesic metric given by
  \begin{equation*}
   d(x,x')
    := \arccos (\langle x,x' \rangle_{\R^3})
  \end{equation*}
for every $x,x' \in \IS^2$, where $\langle \cdot,\cdot \rangle_{\R^3}$ denotes the corresponding Euclidean inner product.
Furthermore, let $(\gO, \cA, \IP)$ be a probability space and $T$ a $2$-weakly iGRF on~$\IS^2$.
Then, by~\cite[Theorem~5.13]{MP11}, $T$ admits an expansion with respect to the \emph{surface spherical harmonic functions} 
$\cY := (Y_{\ell m}, \ell \in \N_0, m=-\ell, \ldots, \ell)$ 
as mappings 
$Y_{\ell m}: [0,\pi] \times [0,2\pi) \rightarrow \C$, 
which are given by
  \begin{equation*}
   Y_{\ell m}(\vt, \vp)
    := \sqrt{\frac{2\ell + 1}{4\pi}\frac{(\ell-m)!}{(\ell+m)!}} P_{\ell m}(\cos \vt) e^{im\vp}
  \end{equation*}
for $\ell \in \N_0$, $m = 0,\ldots, \ell$, and by
\begin{equation*}
   Y_{\ell m}
    := (-1)^m \overline{Y_{\ell -m}}
\end{equation*}
for $\ell \in \N$ and $m=-\ell, \ldots,-1$.
Here $(P_{\ell m}, \ell \in \N_0,m=0,\ldots,\ell)$ denote the \emph{associated Legendre functions} which are given by
\begin{equation*}
   P_{\ell m}(\rho)
    := (-1)^m (1-\rho^2)^{m/2} \frac{\partial^m}{\partial \rho^m} P_\ell(\rho)
\end{equation*}
for $\ell \in \N_0$, $m = 0,\ldots,\ell$, and $\rho \in [-1,1]$,
where 
$(P_\ell, \ell \in \N_0)$ are the \emph{Legendre polynomials} 
given by Rodrigues' formula (see, e.g., \cite{Szego})
\begin{equation*}
P_\ell(\rho)
:= 
2^{-\ell} \frac{1}{\ell!} \, \frac{\partial^\ell}{\partial \rho^\ell} (\rho^2 -1)^\ell
\end{equation*}
for all $\ell \in \N_0$ and $\rho \in [-1,1]$.
This expansion of~$T$ converges in $L^2(\gO\times \IS^2)$ as well as for every $x \in \IS^2$ in~$L^2(\gO)$ and
is given by (see, e.g., \cite[Corollary~2.5]{LS15})
  \begin{equation*}
   T
    = \sum_{\ell=0}^\infty \sum_{m=-\ell}^\ell a_{\ell m} Y_{\ell m},
  \end{equation*}
where $\IA := (a_{\ell m}, \ell \in \N_0, m= - \ell, \ldots, \ell)$ 
is a sequence of complex-valued, centered, Gaussian random variables 
with the following properties:
\begin{enumerate}
 \item $\IA_+ := (a_{\ell m}, \ell \in \N_0, m = 0,\ldots,\ell)$ is a  
       sequence of independent, complex-valued Gaussian random variables.
 \item The elements of $\IA_+$ with $m>0$ satisfy  
       $\Re a_{\ell m}$ and $\Im a_{\ell m}$ are independent and $\cN(0,A_\ell/2)$ distributed.
   \item 
   The elements of $\IA_+$ with $m=0$ are real-valued and
   the elements $\Re a_{\ell 0}$ are $\cN(0,A_\ell)$ 
   distributed for $\ell \in \N$ while $\Re a_{00}$ is $\cN(\E(T)2\sqrt{\pi},A_0)$ distributed.
   \item The elements of~$\IA$ with $m <0$ are deduced from those of $\IA_+$
    by the formulae
    \begin{equation*}
    \Re a_{\ell m} = (-1)^m \Re a_{\ell -m},
      \quad
    \Im a_{\ell m} = (-1)^{m+1} \Im a_{\ell -m}.    
    \end{equation*}
  \end{enumerate}
Here $(A_\ell, \ell \in \N_0)$ is called the \emph{angular power spectrum}.

In what follows we set $Y_{\ell m}(y) := Y_{\ell m}(\vt,\vp)$ for $y \in \IS^2$,
where we identify (with a slight abuse of notation) 
Cartesian and angular coordinates by $y := (\sin \vt \cos \vp, \sin \vt \sin \vp, \cos \vt)$,
and we do not separate indices for doubly sub-~or superscripted
functions and coefficients by a comma, with the understanding
that the reader will recognize double indices as such.
Furthermore, we denote by $\gs$ the \emph{Lebesgue measure on the sphere}
which admits the representation
  \begin{equation*}
   \mathrm d\gs(y)
    = \sin \vt \, \mathrm d\vt \, \mathrm d\vp
  \end{equation*}
for $y \in \IS^2$, $y = (\sin \vt \cos \vp, \sin \vt \sin \vp, \cos \vt)$.

We define the \emph{spherical Laplacian}, also called 
\emph{Laplace--Beltrami operator}, in terms of spherical coordinates 
similarly to~\cite[Section~3.4.3]{MP11} by
  \begin{equation*}
   \Delta_{\IS^2}
    := (\sin \vt)^{-1} \frac{\partial}{\partial \vt} \left( \sin \vt \, \frac{\partial}{\partial \vt} \right)
	+ (\sin \vt)^{-2} \frac{\partial^2}{\partial \vp^2}.
  \end{equation*}
It is well-known (see, e.g., \cite[Theorem~2.13]{M98})
that the spherical harmonic functions~$\cY$ 
are the eigenfunctions of~$-\Delta_{\IS^2}$ with 
eigenvalues $(\ell(\ell+1), \ell \in \N_0)$, i.e.,
  \begin{equation*}
   - \Delta_{\IS^2} Y_{\ell m}
    = \ell(\ell+1) Y_{\ell m}
  \end{equation*}
for all $\ell \in \N_0$, $m = -\ell, \ldots,\ell$.
Furthermore, it is shown in~\cite[Theorem~2.42]{M98} that $L^2(\IS^2;\C)$ 
has the direct sum decomposition
  \begin{equation*}
   L^2(\IS^2;\C) = \bigoplus_{\ell=0}^\infty \cH_\ell(\IS^2),
  \end{equation*}
where the spaces $(\cH_\ell, \ell \in \N_0)$ are 
spanned by spherical harmonic functions
  \begin{equation*}
   \cH_\ell(\IS^2)
    := \Span \{ Y_{\ell m}, m = -\ell,\ldots,\ell\},
  \end{equation*}
i.e., $\cH_\ell(\IS^2)$ denotes the space of eigenfunctions of~$-\Delta_{\IS^2}$ 
that correspond to the eigenvalue $\ell(\ell+1)$ for $\ell \in \N_0$.
Let us denote by $L^2(\IS^2)$ the subspace of all real-valued functions of~$L^2(\IS^2;\C)$.
Then, every real-valued function~$f$ in~$L^2(\IS^2;\C)$
admits a spherical harmonics series expansion
\begin{equation}\label{eq:KLExp}
   f = \sum_{\ell=0}^{+\infty} \sum_{m=-\ell}^\ell f_{\ell m} Y_{\ell m},
\end{equation}
and the coefficients satisfy (cp., e.g., ~\cite[Remark~3.37]{MP11})
\begin{equation*}
   f_{\ell m} =  (-1)^m \overline{f_{\ell -m}},
\end{equation*}
i.e., $f$ can be represented in $L^2(\IS^2)$ by the series expansion
\begin{equation}\label{eq:real_exp}
   f 
    = \sum_{\ell=0}^{+\infty} \Bigl( f_{\ell 0} Y_{\ell 0}
    + 2 \sum_{m=1}^\ell \left(\Re f_{\ell m} \Re Y_{\ell m} 
    - \Im f_{\ell m} \Im Y_{\ell m}\right)\Bigr)
.
\end{equation}
We shall be partly concerned with spectral approximations by truncation of the spherical harmonics expansion~\eqref{eq:KLExp}. 
To state results on convergence rates of such truncations, we 
introduce for any truncation levels $L_1< L_2\in\N_0$ the spaces
\begin{equation}\label{eq:YL}
\cH_{L_1:L_2} := \bigoplus_{\ell=L_1}^{L_2} \cH_\ell
\subset L^2(\IS^2;\C) 
\end{equation}
and identify $\cH_{L:L}:=\cH_L$ for any $L\in\N_0$.
Evidently, $\cH_{0:L}$ is a space of finite dimension that satisfies for $L\in\N$ that
\begin{equation}\label{eq:dim_Y_L}
 L^2 \le N_L := \dim(\cH_{0:L}) = (L+1)^2 \le 4 L^2,
\end{equation}
and thus, in particular, is closed. 
For $L\in\N_0$, we denote by 
$\Pi_L : L^2(\IS^2;\C) \rightarrow \cH_{0:L}$ 
the projector on~$\cH_{0:L}$ given by the 
truncated \KL series~\eqref{eq:KLExp}, 
i.e., for $f\in L^2(\IS^2;\C)$,
\begin{equation}\label{eq:PL}
\Pi_L f := \sum_{\ell=0}^L \sum_{|m|\leq \ell} f_{\ell m} Y_{\ell m}.
\end{equation}
To characterize the decay of the coefficients in the expansion~\eqref{eq:KLExp} and, 
 accordingly, also convergence rates of the projections $\Pi_L$ in~\eqref{eq:PL},
 we introduce for a smoothness index $s\in\R$ and $q\in(1,+\infty)$ 
the Sobolev spaces on 
$\IS^2$ as 
\begin{equation*}
  H^s_q(\IS^2):=(\Id - \Delta_{\IS^2})^{-s/2}L^q(\IS^2).
\end{equation*}
Then, for every $f\in H^s_q(\IS^2)$,
  \begin{equation*}
  \|f\|_{H^s_q(\IS^2)}:=\|(\Id-\Delta_{\IS^2})^{s/2}f\|_{L^q(\IS^2)}
  \end{equation*}
defines a norm on $H^s_q(\IS^2)$, where for $s<0$, the elements of $H^s_q(\IS^2)$
have to be understood as distributions (cp.\ ~\cite[Definition~4.1]{strichartz1983AnaLapcomRieman}).
The positive definiteness of this norm is implied by~\cite[Theorem~XI.2.5]{taylor1981PDO}.
For more details on these spaces, we refer the reader to~\cite{strichartz1983AnaLapcomRieman,taylor1981PDO}.
In the case $q=2$ we omit $q$ in our notation and simply write $H^s(\IS^2)$.
In this setting $H^0(\IS^2):= L^2(\IS^2)$ is identified with its dual space $H^0(\IS^2)^*$ 
and $H^s(\IS^2)^* = H^{-s}(\IS^2)$ for every $s>0$.
Since the norm on $H^s_q(\IS^2)$ is well-defined for every $s\in\R$ and every $q\in (1,+\infty)$, 
we obtain that
\begin{equation}\label{eq:inv_Id-Beltrami}
 (\Id - \Delta_{\IS^2})^{s/2}: H^t_q(\IS^2) \rightarrow H^{t-s}_q(\IS^2) 
\end{equation}
is bounded and surjective for every $t\in\R$.

Since $\cY$ diagonalizes $-\Delta_{\IS^2}$ and therefore
\begin{equation}\label{eq:SpLapEig}
 (\Id-\Delta_{\IS^2})^{s/2} Y_{\ell m}
  = (1 + \ell(\ell+1))^{s/2} Y_{\ell m}
\end{equation}
for every $Y_{\ell m} \in \cY$ by the spectral mapping theorem, 
cp.\ ~\cite[Theorem~10.33(a)]{rudin1973FA} 
applied to the bounded inverse of $(\Id-\Delta_{\IS^2})$ on $L^2(\IS^2)$,
we obtain the following approximation result 
of the operator $(\Id-\Pi_L)$.
\begin{proposition}\label{prop:FAppr}
For every $-\infty < s \leq t < +\infty$ and for every $f\in H^t(\IS^2)$,
\begin{equation*}
\| f - \Pi_L f \|_{H^s(\IS^2)} 
\leq 
L^{-(t-s)} 
\| f \|_{H^t(\IS^2)} 
\leq 
2^{t-s}  
N_L^{-(t-s)/2} 
\| f \|_{H^t(\IS^2)} 
\end{equation*}
for every $L \in \N_{0}$.
\end{proposition}
\begin{proof} 
Let $-\infty < s \leq t < +\infty$, and $f\in H^t(\IS^2)$. 
Then, for $L \in \N$, it holds by~\eqref{eq:KLExp} and~\eqref{eq:SpLapEig} that
\begin{align*}
 \|f-\Pi_Lf\|_{H^s(\IS^2)}^2
  &= \sum_{\ell=L+1}^{+\infty} \sum_{m=-\ell}^\ell |f_{\ell m}|^2 
		\left( 1 + \ell(\ell+1) \right)^{s}
\\
  & \le  \sum_{\ell=L+1}^{+\infty} \sum_{m=-\ell}^\ell |f_{\ell m}|^2
		\left( 1 + \ell(\ell+1) \right)^{s} \Bigl(\frac{\ell(\ell+1)}{L(L+1)}\Bigr)^{t-s}\\
  & \le L^{-2(t-s)} \sum_{\ell=L+1}^{+\infty} \sum_{m=-\ell}^\ell |f_{\ell m}|^2 
		\left( 1 + \ell(\ell+1) \right)^t
    \le  L^{-2(t-s)} \|f\|_{H^t(\IS^2)}^2.
\end{align*}
The relation 
between $L$ and $N_L$ in~\eqref{eq:dim_Y_L} implies the assertion.
\end{proof}

%
%

Let us next introduce H\"older spaces on~$\IS^2$.
For $\iota\in \N_0$, we denote by $C^\iota(\IS^2)$ the space of $\iota$-times
continuously differentiable functions taking values in $\R$ 
and, 
for $\gg\in(0,1)$, by $C^{\iota,\gg}(\IS^2)\subset C^\iota(\IS^2)$ 
the subspace of functions whose $\iota$-th derivative is H\"older continuous
with exponent $\gg$. We identify $C^{\iota,0}(\IS^2)$ with $C^{\iota}(\IS^2)$.
The H\"older spaces satisfy the Sobolev embedding that 
\begin{equation*}
 H^s_q(\IS^2)\subset C^{\iota,\gg}(\IS^2)
\end{equation*}
 is continuously embedded for $s-2/q\geq\iota+\gg$, $\gg\neq 0$, which is stated 
 for~$\IS^2$ in Theorem~\ref{thm:Sobelev_emb}.

As final functional analytical ingredient, we 
need $L^p$ spaces on the probability space with values in a 
Banach space to consider integrability of iGRFs as H\"older-space-valued random variables.
Therefore, let $(B,\|\cdot\|_B)$ denote a Banach space.
For $p\in[1,+\infty)$, the \emph{Bochner space} $L^p(\gO;B)$ consists of 
all strongly $B$-measurable functions $X:\gO\rightarrow B$ such that 
$\|X\|_B$ is in $L^p(\gO)$, i.e.,
\begin{equation*}
 \|X\|_{L^p(\gO;B)}^p := \Bigl(\E\bigl(\|X\|_B^p\bigr)\Bigr)
 < + \infty.
\end{equation*}
Then $(L^p(\gO;B),\|\cdot\|_{L^p(\gO;B)})$ is a Banach space by~\cite[Theorem~III.6.6]{dunford1958LinOpeIGenThe}.
To connect the already introduced convergence of \KL expansions of iGRF with Bochner spaces, we observe that $L^2(\gO\times\IS^2)$ and $L^2(\gO;L^2(\IS^2))$ are isometrically isomorphic, i.e., the \KL expansion also converges in $L^2(\gO;L^2(\IS^2))$ and 
$\|T\|_{L^2(\gO;L^2(\IS^2))}^2  = \sum_{\ell=0}^{+\infty} A_\ell \frac{2\ell+1}{4\pi}$ is finite.
For more details on the functional analytical setting and measurability, the reader is referred to Appendix~\ref{app:meas_func-ana}. 

Let us now return to the isotropic Gaussian random field~$T$ and assume from here on that
\begin{equation}\label{eq:AngSpecDec}
 \sum_{\ell=0}^{+\infty} A_\ell \ell^{1+\gb}<+\infty
\end{equation}
for some $\gb>0$.
It was shown in~\cite[Theorem~4.6]{LS15} that this condition yields the existence of a modification of~$T$ that is in~$C^{\iota,\gg}(\IS^2)$ for all $\iota + \gg < \gb/2$, which we consider from now on without loss of generality. The purpose of the following theorem is to show strong measurability, $L^p$-integrability, and approximation of this iGRF. We remark that the proof of the theorem just requires a continuous modification of~$T$, which exists by~\cite[Theorem~4.5]{LS15}, and therefore recovers \cite[Theorem~4.6]{LS15} from \cite[Theorem~4.5]{LS15} with a possibly different modification. This follows since $L^p$ integrability holds only if $T \in C^{\iota,\gg}(\IS^2)$ $\IP$-a.s..

\begin{theorem} 
\label{thm:LpRegIsoGRF}
Let $T$ be a continuous iGRF that satisfies~\eqref{eq:AngSpecDec} for some $\gb>0$.
Then, for every $p\in[1,+\infty)$, $\iota\in\N_0$, 
and $\gg\in(0,1)$ with $\iota+\gamma<\gb/2$, it holds that $T\in L^p(\Omega;C^{\iota,\gamma}(\IS^2))$.
Furthermore, there exists a constant $C_{p,\iota,\gg}$, which is 
independent of $(A_\ell, \ell \in \N_0)$ such that for every $L\in\N_{0}$,
\begin{equation*}
 \|T-\Pi_L T\|_{L^p(\Omega;C^{\iota,\gamma}(\IS^2))}
 \leq C_{p,\iota,\gg} \Bigl( \sum_{\ell>L}A_\ell \ell^{1+\gb}\Bigr)^{1/2}.
\end{equation*}
\end{theorem}
\begin{proof}
 It suffices to prove the theorem for $p$ even, i.e., 
 for $p = 2p'$ and $p' \in \N$. 
 The result for all remaining $p \in [1, + \infty)$ follows then by H\"older's inequality.
 We set $T^L:=\Pi_L T$
 and show first that $(T^L,L\in\N_{0})$ is a Cauchy sequence in 
 $L^p(\Omega;C^{\iota,\gamma}(\IS^2))$.
 The smoothness of the spherical harmonics implies with Pettis' theorem
 (see Theorem~\ref{thm:Pettis}) that $T^L$ is strongly measurable in every function 
 space that contains $C^\infty(\IS^2)$, $L\in\N_{0}$.
 In particular, $T^L$ is strongly $B$-measurable, 
$B\in\{C^{\iota,\gg}(\IS^2),H^{\gb/2}_q(\IS^2),q\in(1,+\infty)\}$ for every $L\in\N_{0}$.
 With the identity $\sum_{|m| \leq \ell} \left| Y_{\ell m}(x) \right|^2 = (2\ell+1)/(4\pi)$ (cp.\ ~\cite[Theorem~2.4.5]{Ned2001})
 and the \KL expansion, 
 we observe that $\sum_{m=-\ell}^\ell a_{\ell m} Y_{\ell m}(x)$ 
 is $\cN(0,(2\ell+1)/(4\pi) \, A_\ell)$-distributed for every $\ell \in \N$ 
 and $x \in \IS^2$ as well as that $(\sum_{m=-\ell}^\ell a_{\ell m} Y_{\ell m}(x) ,\ell\in\N_0)$ 
is a sequence of independent random variables for every fixed $x \in \IS^2$.
 Hence, for $L_1 > L_2\in\N_{0}$, we obtain that
 \begin{align*}
  \|T^{L_1}-T^{L_2}\|_{L^{2p'}(\Omega;H^{\gb/2}_{2p'}(\IS^2))}^{2p'}
  &=\int_{\IS^2}\E\Bigl(
	\bigl(\sum_{\ell=L_2+1}^{L_1}\sum_{m=-\ell}^\ell 
	    a_{\ell m} (1+\ell(\ell+1))^{\gb/4}Y_{\ell m}
	\bigr)^{2p'}\Bigr)\dd\sigma
   \\
  &= \frac{(2p')!}{2^{p'}p'!}  |\IS^2| \Bigl(\sum_{\ell=L_2+1}^{L_1} 
               A_\ell \frac{2\ell+1}{4\pi}(1+\ell(\ell+1))^{\gb/2}
    \Bigr)^{p'}
  < + \infty,
 \end{align*}
  where we applied Fubini's theorem and the fact that moments of centered Gaussian random variables satisfy
 $\E(X^{2p'}) =(2p')!/(2^{p'}p'!) \E(X^2)^{p'}$. Finiteness follows since~\eqref{eq:AngSpecDec} holds.
 
 This implies especially with the Sobolev embedding (cp.\ Theorem~\ref{thm:Sobelev_emb}) that there exists a constant~$C$ such that
 \begin{equation*}
  \|T^{L_1}-T^{L_2}\|_{L^{2p'}(\Omega;C^{\iota,\gg}(\IS^2))}
  \leq C \Bigl(4\pi \frac{(2p')!}{2^{p'}p'!} \Bigr)^{1/(2p')}
      \Bigl(\sum_{\ell=L_2+1}^{L_1} A_\ell \frac{2\ell+1}{4\pi}(1+\ell(\ell+1))^{\gb/2} \Bigr)^{1/2}
 \end{equation*}
 for $\gb/2 - 1/p'\geq \iota+\gg$ and therefore that $(T^L,L\in\N_0)$ is a Cauchy sequence in $L^{2p'}(\Omega;C^{\iota,\gg}(\IS^2))$ 
 that converges due to completeness. Furthermore, the result extends by H\"older's inequality to $L^{p}(\Omega;C^{\iota,\gg}(\IS^2))$ 
 for every $p \le 2p'$. 
Since $L^p$ limits are $\IP$-almost surely unique 
and we know by the properties of the \KL expansion 
that $(T^L,L\in\N_0)$ converges to~$T$ 
in~$L^2(\Omega;L^2(\IS^2))$, $T \in L^{2p'}(\Omega;C^{\iota,\gg}(\IS^2))$ 
holds also 
due to the assumed continuity.
 
For given $p \ge 1$, we choose $p' \in \N$ such that 
$p\le 2p'$ and $\gb/2 - 1/p'\geq \iota+\gg$ for fixed $\iota$ and~$\gg$.
This implies that there exists a constant~$C_{p'}$, i.e., $C_{p,\iota,\gg}$, such that
\begin{equation*}
 \|T^{L_1}-T^{L_2}\|_{L^{p}(\Omega;C^{\iota,\gg}(\IS^2))}
 \leq C_{p,\iota,\gg} \Bigl(\sum_{\ell=L_2+1}^{L_1}A_\ell \ell^{1+\gb}\Bigr)^{1/2}
.
\end{equation*}
We obtain the claim by taking the limit $L_1 \to + \infty$.
\end{proof}
Let us continue with the properties of the corresponding isotropic lognormal random fields $a:=\exp(T)$ given by $a(x):=\exp(T(x))$ for every $x \in \IS^2$. These will be of interest as diffusion coefficients of the elliptic operators in our considered SPDEs.
For the approximation of these lognormal random fields, we set similarly $a^L:=\exp(\Pi_L T)$ for every $L\in\N_0$.
Then, the properties of~$T$ and~$T^L$ shown in Theorem~\ref{thm:LpRegIsoGRF} imply similar results for~$a$ and~$a^L$, which are stated in the following theorem.
\begin{theorem}\label{thm:logNiGRFsProp}
Let $a=\exp(T)$ be an isotropic lognormal RF such that $T$ is a continuous iGRF 
and satisfies~\eqref{eq:AngSpecDec} for some $\gb>0$.
Then, for every $p\in[1,+\infty)$, $\iota\in\N_0$, and for $\gg\in(0,1)$ 
satisfying $\iota+\gamma<\gb/2$, 
and for every $L\in\N_0$, it holds that $a, a^L\in L^p(\Omega;C^{\iota,\gamma}(\IS^2))$,
where the $L^p(\Omega;C^{\iota,\gamma}(\IS^2))$-norm of $a^L$ can be bounded independently of $L$
and the same stays true for $a, a^L\in L^p(\Omega;C^0(\IS^2))$.
Furthermore, for every $\varepsilon\in(0,\gb)$, 
there exists a constant $C_{p,\varepsilon}$ such that for every $L\in\N_0$, it holds that
\begin{equation*}
 \|a-a^L\|_{L^p(\Omega;C^{0}(\IS^2))}
 \leq C_{p,\varepsilon} \Bigl( \sum_{\ell>L}A_\ell \ell^{1+\varepsilon}\Bigr)^{\frac{1}{2}}.
\end{equation*}
\end{theorem}

\begin{proof}
We observe first that the composition with the exponential function is a continuous mapping from $C^{\iota,\gg}(\IS^2)$
into itself and $T$ is strongly $C^{\iota,\gg}(\IS^2)$-measurable by Theorem~\ref{thm:LpRegIsoGRF}.
Then, the inequality
\begin{equation*}
 \|\exp( v)\|_{C^{\iota,\gg}(\IS^2)}\leq C_{\iota,\gg} \|\exp(v)\|_{C^{0}(\IS^2)}\bigl(1+\|v\|_{C^{\iota,\gg}(\IS^2)}^{\iota+1}\bigr),
\end{equation*}
which follows in a similar way as the proof of~\cite[Theorem~A.8]{Hoermander1976} and which is provenin Lemma~\ref{lemma:Hoelder_prod_comp_est},
implies strong $C^{\iota,\gg}(\IS^2)$-measurability of $a=\exp(T)$ and of $a^L=\exp(\Pi_L T)$ 
for every $L\in\N_0$.
The Cauchy--Schwarz inequality then implies
that there exists a constant $C$ 
that does not depend on $T$
such that 
\begin{equation*}
\|a\|_{L^p(\gO;C^{\iota,\gg}(\IS^2))}=\|\exp(T)\|_{L^p(\gO;C^{\iota,\gg}(\IS^2))}
\leq C \|\exp(T)\|_{L^{2p}(\gO;C^{0}(\IS^2))}(1+\|T\|_{L^{2p(\iota+1)}(\gO;C^{\iota,\gg}(\IS^2))}^{\iota+1}).
\end{equation*}
The second term in the product is bounded by Theorem~\ref{thm:LpRegIsoGRF}, while the boundedness of the first one is a consequence of Fernique's theorem, which is proven in a similar way as \cite[Proposition~3.10]{Charrier_SINUM12} and can be found for iGRFs on~$\IS^2$ in Proposition~\ref{prop:iLogRF_C^0bound}.

The second assertion about $a^L$ is proven completely analogously and the 
$L^p(\Omega;C^{\iota,\gamma}(\IS^2))$-norm of $a^L$ 
can be bounded independently of $L$ due to 
Theorem~\ref{thm:LpRegIsoGRF} and the independence of~$L$ in the $L^p(\Omega;C^0(\IS^2))$-norm, which is also part of Proposition~\ref{prop:iLogRF_C^0bound}.

For the proof of the third claim, note that the fundamental theorem of calculus implies 
for arbitrary $t,s\in\R$ that $|\exp(t)-\exp(s)|\leq (\exp(t)+\exp(s))|t-s|$,
which yields with 
the Cauchy--Schwarz inequality that
\begin{equation*}
 \|a-a^L\|_{L^p(\gO;C^0(\IS^2))}\leq \bigl(\|a\|_{L^{2p}(\gO;C^0(\IS^2))} + \|a^L\|_{L^{2p}(\gO;C^0(\IS^2))}\bigr) 
 \|T-\Pi_L T\|_{L^{2p}(\gO;C^0(\IS^2))}.
\end{equation*}
Therefore, the third assertion follows with Theorem~\ref{thm:LpRegIsoGRF}.
 \end{proof}

In the following and especially in the analysis of~\eqref{eq:ellSPDE}, the properties of the minimum and the maximum of a random field are of major interest. Therefore,
we define for $a=\exp(T)$, where $T$ is a continuous iGRF~$T$, the random variables
\begin{equation*}
 \hat{a}:=\max_{x\in\IS^2}a(x)
 \quad \text{and}\quad 
 \quad\check{a} :=\min_{x\in\IS^2}a(x),
 \end{equation*}
 and similarly for $L \in \N_0$
 \begin{equation*}
 \hat{a}^L:=\max_{x\in\IS^2}a^L(x)
 \quad \text{and}\quad 
 \check{a}^L :=\min_{x\in\IS^2}a^L(x).
\end{equation*}
Here we recall that $a^L=\exp(\Pi_L T)$.
Since
 \begin{equation*}
 \|\check{a}^{-1}\|_{L^p(\gO)}
  =\|(\min_{x\in\IS^2}a(x))^{-1}\|_{L^p(\gO)}
   =\|\max_{x\in\IS^2}\; \exp(-T(x))\|_{L^p(\gO)}
   =\| \exp(- T) \|_{L^p(\gO;C^0(\IS^2))}
 \end{equation*}
 and
 \begin{equation*}
  \|\hat{a}\|_{L^p(\gO)}
    = \|\exp(T)\|_{L^p(\gO;C^0(\IS^2))},
 \end{equation*}
these are elements of $L^p(\gO)$, $p\in[1,+\infty)$ by Theorem~\ref{thm:logNiGRFsProp}, which is summarized in the following corollary.
\begin{corollary}\label{cor:apint}
Let $T$ be a continuous iGRF,
then $\hat{a}$, $\check{a}^{-1}$, $\hat{a}^L$, and $(\check{a}^L)^{-1}$ are in $L^p(\gO)$ 
for every $p\in[1,+\infty)$ and every $L\in\N_0$, where the $L^p(\gO)$-norm
of $\hat{a}^L$ and $(\check{a}^L)^{-1}$ can be bounded independently of $L$.
\end{corollary}

\section{Existence, uniqueness, and regularity of solutions}
\label{sec:ExUniqRegSol}

Having introduced the analytic and approximation properties of the random source of interest, we are now in state to come back to the SPDE of interest
\begin{equation}
 -\nabla_{\IS^2}\cdot(a\nabla_{\IS^2}u)=f,
 \tag{\ref{eq:ellSPDE}}
\end{equation}
where $a=\exp(T)$ is an isotropic lognormal random field such that the iGRF~$T$ 
is continuous and satisfies~\eqref{eq:AngSpecDec} for some 
$\gb>0$ and $f$ is a deterministic source term which has at least $H^{-1}(\IS^2)$ regularity.

In what follows we first introduce the variational framework in which we consider solutions before we show existence, uniqueness, and regularity of solutions where the latter depends on the regularity of $a$ and $f$. We derive similar results for the SPDEs corresponding to the approximate random fields $a^L$.

We observe first that solutions of the SPDE on the closed, compact submanifold~$\IS^2$ of~$\R^3$ without boundary may exhibit nonuniqueness since $-\nabla_{\IS^2}\cdot(a\nabla_{\IS^2})$ might have a nontrivial kernel, 
i.e., a constant~$u$ is a solution of the homogeneous equation.

Therefore, we shall work in 
factor spaces of function spaces which are orthogonal 
(in $L^2(\IS^2)$) to constants.
The closed subspace of $H^1(\IS^2)$ that consists of all $v\in H^1(\IS^2)$ 
whose inner product with $1$ satisfies $(v,1)=0$
is denoted by $H^1(\IS^2)/\R$. 
For every $v\in H^1(\IS^2)/\R$, 
\begin{equation*}
 \|v\|_{H^1(\IS^2)/\R}:=\|\nabla_{\IS^2}v\|_{L^2(\IS^2)}
\end{equation*}
defines a norm on $H^1(\IS^2)/\R$
due to the second Poincar\'{e} inequality
\begin{equation*}
 \|v\|_{L^2(\IS^2)}\leq \frac{1}{\sqrt{2}}\|\nabla_{\IS^2}v\|_{L^2(\IS^2)},
\end{equation*}
which is proven considering the Reyleigh quotient and the spectrum of~$-\Delta_{\IS^2}$ (for details, see \cite[Lemma~8.3]{Herrmann13}).
Since $H^1(\IS^2)/\R$ is a closed linear subspace of $H^1(\IS^2)$ and the norm $\|\cdot\|_{H^1(\IS^2)/\R}$ 
is induced by the inner product $(\nabla_{\IS^2}\cdot,\nabla_{\IS^2}\cdot)$, $H^1(\IS^2)/\R$ is a Hilbert space.

Let us consider the variational formulation of~\eqref{eq:ellSPDE} in $H^1(\IS^2)/\R$ 
with right hand side $f\in H^{-1}(\IS^2)$ such that $f(1)=0$:
find a strongly $H^1(\IS^2)/\R$-measurable mapping $u$ such that
\begin{equation}\label{eq:ellSPDEweak}
(a \nabla_{\IS^2} u, \nabla_{\IS^2} v)  
= 
f(v) \quad \forall v\in H^1(\IS^2)/\R.
\end{equation}
Moreover, we want to show that this mapping $u:\gO\rightarrow H^1(\IS^2)/\R$ is $L^p$-integrable.
To this end, let us fix this right hand side $f$.
In what follows let us first recall the deterministic existence and uniqueness theory and derive the results in such a form that they are suitable for the stochastic framework. These will then be applied to~\eqref{eq:ellSPDEweak}.
Therefore, let
\begin{equation*}
C^0_+(\IS^2) := \{\tilde{a} \in C^0(\IS^2), \min\limits_{x\in\IS^2}\tilde{a}(x)>0\}. 
\end{equation*}
and consider the corresponding deterministic variational problem for $\tilde{a}\in C^0_+(\IS^2)$ with right hand side $f\in H^{-1}(\IS^2)$ such that $f(1)=0$: 
find $u \in H^1(\IS^2)/\R$ such that
\begin{equation}\label{eq:ellPDEweak}
(\tilde{a} \nabla_{\IS^2} \tilde{u}, \nabla_{\IS^2} v)  
= 
f(v) \quad \forall v\in H^1(\IS^2)/\R.
\end{equation}
Since the bilinear form
$ (\tilde{a} \nabla_{\IS^2} \cdot, \nabla_{\IS^2} \cdot)  $ 
is continuous and coercive on the space $H^1(\IS^2)/\R\times H^1(\IS^2)/\R$, i.e.,
\begin{equation}\label{eq:bilinearform_continuous}
(\tilde{a} \nabla_{\IS^2} v, \nabla_{\IS^2} w)
  \leq \|\tilde{a}\|_{C^0(\IS^2)}\|v\|_{H^1(\IS^2)/\R} \|w\|_{H^1(\IS^2)/\R} 
  \quad \forall v,w\in H^1(\IS^2)/\R
\end{equation}
and
\begin{equation}\label{eq:bilinearform_coercive}
\|v\|^2_{H^1(\IS^2)/\R} 
\leq \frac{1}{\min_{x\in\IS^2}\tilde{a}(x)} (\tilde{a} \nabla_{\IS^2} v, \nabla_{\IS^2} v) \quad \forall v\in H^1(\IS^2)/\R,
\end{equation}
existence and uniqueness of a solution $\tilde{u} \in H^1(\IS^2)/\R$ to~\eqref{eq:ellPDEweak} as well as the estimate
\begin{equation}\label{est:Pth_H1}
\|\tilde{u}\|_{H^1(\IS^2)/\R}
\leq 
\frac{1}{\min_{x\in\IS^2}\tilde{a}(x)}\sqrt{\frac{3}{2}}\|f\|_{H^{-1}(\IS^2)}
,
\end{equation}
are implied by the Lax--Milgram lemma, where we used that
\begin{equation*}
 \sup_{0\neq v\in H^1(\IS^2)/\R} |f(v)|/\|v\|_{H^1(\IS^2)/\R} \leq \sqrt{\frac{3}{2}} \|f\|_{H^{-1}(\IS^2)}.
\end{equation*}


The difference of two solutions with respect to different coefficients~$\tilde{a}$ and the same right hand side~$f$ 
can be estimated with a 
version of Strang's second lemma. 
This is made precise in the following lemma, where 
the variational formulation~\eqref{eq:ellPDEweak} 
is also considered with respect to subspaces
of~$H^1(\IS^2)/\R$ to be suitable for approximations in Section~\ref{sec:Discr}. The proof for $H^1(\IS^2)/\R$ can be found in~\cite[Proposition~8.6]{Herrmann13} (with a different norm on $f$) which also applies for proper, closed subspaces of~$H^1(\IS^2)/\R$.
\begin{lemma}\label{lemma:Strang}
Let $V\subset H^1(\IS^2)/\R$ be a closed, 
not necessarily strict subspace of $H^1(\IS^2)/\R$ endowed with the $H^1(\IS^2)/\R$-norm.
For $\tilde{a}_1,\tilde{a}_2\in C^0_+(\IS^2)$,
let $\tilde{u}_1,\tilde{u}_2 \in V$ satisfy 
\begin{equation*}
  (\tilde{a}_i \nabla_{\IS^2} \tilde{u}_i, \nabla_{\IS^2} v)  
= 
f(v)
\quad \forall v\in V
\end{equation*}
for $i=1,2$. Then,
 \begin{equation*}
 \|\tilde{u}_1 - \tilde{u}_2\|_{H^1(\IS^2)/\R}
 \leq 
\sqrt{\frac{3}{2}}\frac{\|f\|_{H^{-1}(\IS^2)}}{(\min_{x\in\IS^2}\tilde{a}_1(x))(\min_{x\in\IS^2}\tilde{a}_2(x))}
 \|\tilde{a}_1 - \tilde{a}_2\|_{C^0(\IS^2)}
 .
 \end{equation*}
\end{lemma}
Let us denote the solution map that maps the coefficient $\tilde{a}\in C^0_+(\IS^2)$
to the respective unique solution $\tilde{u} \in H^1(\IS^2)/\R$ of~\eqref{eq:ellPDEweak} by
\begin{equation}\label{eq:solution_map}
 \Phi_{{f}} : 
 C^0_+(\IS^2)
 \rightarrow 
 H^1(\IS^2)/\R
 ,
\end{equation}
then we obtain the following proposition as a direct consequence of the previous lemma.
\begin{proposition}\label{prop:solution_map_cont}
 $\Phi_{f} : 
 C^0_+(\IS^2)
 \rightarrow 
 H^1(\IS^2)/\R$ 
 is continuous.
\end{proposition}
We now state the well-posedness of the weak formulation of the SPDE~\eqref{eq:ellSPDEweak}.
\begin{theorem}\label{thm:ExEllSPDE}
Let $a = \exp(T)$ be an isotropic lognormal RF such that the iGRF 
$T$ is continuous and satisfies~\eqref{eq:AngSpecDec} for some $\gb>0$.
Then, there exists $u\in L^p(\Omega;H^1(\IS^2)/\R)$ for every $p\in[1,+\infty)$ 
such that $u$ is in this sense the unique solution of~\eqref{eq:ellSPDEweak}.
\end{theorem}
\begin{proof}
Since $a$ takes values in $C^0_+(\IS^2)$, we set $u := \Phi_{f}(a)$,
which solves~\eqref{eq:ellSPDEweak} uniquely.
The continuity of $\Phi_{f}$ in Proposition~\ref{prop:solution_map_cont} implies strong $H^1(\IS^2)/\R$-measurability (cp.\ ~Lemma~\ref{lemma:comp_meas_maps}) and $L^p$-integrability follows with~\eqref{est:Pth_H1} and Corollary~\ref{cor:apint}.
\end{proof}
Since the computation of the random coefficient $a=\exp(T)$ 
does not seem to be feasible in general due to the infinite \KL expansion of~$T$,
we consider solutions with respect to the coefficients $(a^L,L\in\N_0)$ in what follows and 
analyze the convergence of the resulting sequence of solutions in 
$L^p(\gO;H^1(\IS^2)/\R)$, $p\in[1,+\infty)$.
For every $L\in\N_0$, we consider the variational problem:
find a strongly $H^1(\IS^2)/\R$-measurable mapping $u^L$ such that
\begin{equation}\label{eq:ellSPDEweak_truncated}
(a^L \nabla_{\IS^2} u^L, \nabla_{\IS^2} v)  
= 
f(v) \quad \forall v\in H^1(\IS^2)/\R.
\end{equation}
This is a special case of Theorem~\ref{thm:ExEllSPDE}, which implies existence, uniqueness, and $L^p$-integrability of a solution~$u^L$. It is clear from Corollary~\ref{cor:apint} that the $L^p$-norm can be bounded uniformly in~$L$. We state the result for further use in the following corollary.

\begin{corollary}
Let the assumptions of Theorem~\ref{thm:ExEllSPDE} be satisfied.
For every $L\in\N_0$, there exists a unique $u^L$ such that
$u^L$ solves~\eqref{eq:ellSPDEweak_truncated} as well as its
$L^p(\Omega;H^1(\IS^2)/\R)$-norm is finite for every $p\in[1,+\infty)$ 
and can be bounded uniformly in~$L$.
\end{corollary}
We conclude the part on existence and uniqueness of solutions with a convergence result that the sequence of solutions $(u^L,L\in\N_0)$ of~\eqref{eq:ellSPDEweak_truncated} converges in~$L^p(\Omega;H^1(\IS^2)/\R)$ to the solution~$u$ of~\eqref{eq:ellSPDEweak}.
\begin{proposition}\label{prop:EllSPDE_approx}
 Let the assumptions of Theorem~\ref{thm:ExEllSPDE} be satisfied.
 Furthermore, let $u$ be the unique solution of~\eqref{eq:ellSPDEweak} and $(u^L,L\in\N_0)$ 
 be the sequence of unique solutions of~\eqref{eq:ellSPDEweak_truncated}.
 Then, for every $p\in[1,+\infty)$ and $\varepsilon\in(0,\gb)$,
 there exists a constant $C_{p,\varepsilon}$ such that for every $L\in \N_0$, 
 it holds that
 \begin{equation*}
  \|u-u^L\|_{ L^p(\Omega;H^1(\IS^2)/\R)}
  \leq 
  C_{p,\varepsilon} \Bigl ( \sum_{\ell>L}A_{\ell} \ell^{1 + \varepsilon}\Bigr)^{1/2}
  .
 \end{equation*}
\end{proposition}
\begin{proof}
For every $L\in\N_0$, a twofold application of H\"older's inequality implies the claim with 
Lemma~\ref{lemma:Strang}, Corollary~\ref{cor:apint}, Theorem~\ref{thm:logNiGRFsProp}, and~\eqref{est:Pth_H1},
i.e., there exists a constant $C_{p,\varepsilon}$ such that for every $L\in\N_0$, 
it holds that
\begin{align*}
 \|u - u^L\|_{L^p(\gO;H^1(\IS^2)/\R)}
 &\leq \sqrt{\frac{3}{2}}\|f\|_{H^{-1}(\IS^2)}
 \|1/\check{a}\|_{L^{3p}(\gO)}\|1/\check{a}^L\|_{L^{3p}(\gO)}
 \|a-a^L\|_{L^{3p}(\gO;C^0(\IS^2))}
 \\
 &\leq C_{p,\varepsilon} \Bigl ( \sum_{\ell>L}A_{\ell} \ell^{1 + \varepsilon}\Bigr)^{1/2}
.\qedhere
\end{align*}
\end{proof}
%
%
Since the goal of this manuscript is to derive high order approximations of the solution~$u$ of~\eqref{eq:ellSPDEweak} with
Finite Element and Spectral Methods, higher order regularity of~$u$ is essential.
In what follows we show that $u$ takes values in 
$H^{1+s}(\IS^2)$ for $s>0$ such that the range of~$s$ is only limited by the 
regularity of~$a$ and the right hand side~$f$.
As before we first consider the regularity of the solution~$\tilde{u}$ of the deterministic problem~\eqref{eq:ellPDEweak} in terms of the solution map~\eqref{eq:solution_map} before applying it to the stochastic framework.
We remark that the domain of $\Phi_{f}$ reflects the regularity of the coefficient $\tilde{a}$ while the range of $\Phi_{f}$ 
reflects the regularity of the respective solution~$\tilde{u}$.

\begin{proposition}\label{prop:solution_map_regularity}
 Let $\iota\in\N_0$, $\gg\in(0,1)$, and $s\in[0,+\infty)$ satisfy $s<\iota+\gg$. If $f\in H^{-1+s}(\IS^2)$, 
 then
 \begin{equation*}
 \Phi_{f} : 
 C^{\iota,\gg}(\IS^2)\cap C^0_+(\IS^2)
 \rightarrow 
 H^{1+s}(\IS^2)
 \end{equation*}
 is continuous with respect to the topology of $C^{\iota,\gg}(\IS^2)$.
 
 Moreover the $H^{1+s}(\IS^2)$-norm can be bounded by the following recursion.
 For $s < 1$, it holds that
 \begin{equation*}
  \|\Phi_{f}(\tilde{a})\|_{H^{1+s}(\IS^2)} 
  \leq 
  C 
  \|\tilde{a}\|_{C^{0,\gamma}(\IS^2)} \|1/\tilde{a}\|_{C^0(\IS^2)}^2
  \|f\|_{H^{-1+s}(\IS^2)}
 .
\end{equation*}
 If $s\geq 1$, then for every $n\in\{0,\dots,\lfloor s \rfloor-1\}$, there exists a constant $C>0$ 
 such that for every $\tilde{a}\in C^{\iota,\gg}(\IS^2)\cap C^0_+(\IS^2)$,
 \begin{align*}
 &\|\Phi_{f}(\tilde{a})\|_{H^{1+(n+1)+\{s\}}(\IS^2)}\\
  & \qquad \leq C \|1/\tilde{a}\|_{C^{n,\gg}(\IS^2)} \bigl( \|f\|_{H^{1+(n-1)+\{s\}}(\IS^2)}
 + \|\tilde{a}\|_{C^{n+1,\gg}(\IS^2)}\|\Phi_{f}(\tilde{a})\|_{H^{1+n+\{s\}}(\IS^2)}\bigr)
\end{align*}
 where $\{s\}$ denotes the fractional part of $s$.
\end{proposition}

While the base case for $s<0$ is proven by the translation of 
results on domains in Euclidean space in~\cite{ChScTe_SINUM13}, higher order regularity is shown by induction  with a perturbation argument. The detailed proof can be found in Appendix~\ref{app:proof_reg_est}.

The proposition transfers to the stochastic framework and enables us to prove the main result of this section to obtain higher order approximations in the following Section~\ref{sec:Discr}.
\begin{theorem}\label{thm:PthReg_Sobolev}
Let $a=\exp(T)$ be an isotropic lognormal RF such that the iGRF $T$ 
is continuous and satisfies~\eqref{eq:AngSpecDec} for some $\gb>0$.
Furthermore, let $u$ be the solution of~\eqref{eq:ellSPDEweak} and $(u^L,L\in\N_0)$ 
be the sequence of solutions of~\eqref{eq:ellSPDEweak_truncated}.
Then, for every  $s\in[0,\gb/2)$ and $L\in\N_0$, it holds that 
$u,u^L\in L^p(\Omega;H^{1+s}(\IS^2))$ for every $p\in[1,+\infty)$,
if $f\in H^{-1+s}(\IS^2)$. 
Moreover the $L^p(\Omega;H^{1+s}(\IS^2))$-norm 
of $u^L$ can be bounded uniformly in~$L$.
\end{theorem}
\begin{proof}
 Let us write $s=\lfloor s \rfloor + \{s\}$, where $\{s\}\in[0,1)$ is the fractional part of $s$,
 and then set $\iota:=\lfloor s \rfloor \in\N_0$ and choose $\gg\in(\{s\}, \min\{\gb/2-\iota,1\})$,
 which implies that $s<\iota+\gg$.
 We deduce that $a,a^L \in L^{p'}(\gO;C^{\iota,\gg}(\IS^2))$ for every $L\in\N_0$, 
$p'\in[1,+\infty)$, from Theorem~\ref{thm:logNiGRFsProp}.
 In particular, these RFs are strongly $C^{\iota,\gg}(\IS^2)$-measurable and positive.
 Hence, by the continuity of the solution map $\Phi_{f}$ from Proposition~\ref{prop:solution_map_regularity}, the mappings
 $u=\Phi_{f}(a)$ and $u^L=\Phi_{f}(a^L)$
 are strongly $H^{1+s}$-measurable for every $L\in\N_0$ (cp.\ ~Lemma~\ref{lemma:comp_meas_maps}).
 
 The boundedness of the $L^p(\Omega;H^{1+s}(\IS^2))$-norm will be proved inductively.
 As a base case we apply the base case estimate of the $H^{1+\{s\}}(\IS^2)$-norm of $u$
 from Proposition~\ref{prop:solution_map_regularity}
 and use the Cauchy--Schwarz inequality to obtain 
 that there exists a constant $C>0$ (independent of $u$, $a$, and $f$)
 such that
 \begin{equation*}
 \|u\|_{L^p(\gO;H^{1+\{s\}}(\IS^2))}
 \leq
 C \|f\|_{H^{-1+\{s\}}(\IS^2)} \|a\|_{L^{2p}(\gO;C^{0,\gg}(\IS^2))} 
\|\check{a}^{-2}\|_{L^{2p}(\gO)}.
 \end{equation*}
We infer from Theorem~\ref{thm:logNiGRFsProp} and Corollary~\ref{cor:apint} 
that the right hand side of the previous inequality
is finite.
Let us assume as induction hypothesis that the 
$L^p(\gO;H^{1+n+\{s\}}(\IS^2))$-norm of $u$ is finite for every 
$n\in\{0,1,\dots,\lfloor s \rfloor -1\}$,
 which we just established for $n=0$.
 Let $n\in\{0,1,\dots,\lfloor s \rfloor -1\}$ and
 let us apply the recursion formula on the $H^{1+(n+1)+\{s\}}(\IS^2)$-norm from Proposition~\ref{prop:solution_map_regularity}
 and apply the Cauchy--Schwarz inequality twice 
 to obtain that there exists a constant $C$ that is independent of $u$, $a$, and $f$ such that
 \begin{align*}
 & \|u\|_{L^p(\gO;H^{1+(n+1)+\{s\}}(\IS^2))}\\
 & \qquad\leq C \|1/a\|_{L^{3p}(\gO;C^{n,\gg}(\IS^2))} \bigl( \|f\|_{H^{1+(n-1)+\{s\}}(\IS^2)}
 + \|a\|_{L^{3p}(\gO;C^{n+1,\gg}(\IS^2))}\|u\|_{L^{3p}(\gO;H^{1+n+\{s\}}(\IS^2))}\bigr).
 \end{align*}
 Since $1/a=\exp(-T)$, Theorem~\ref{thm:logNiGRFsProp} is applicable to
 $-T$, which satisfies~\eqref{eq:AngSpecDec}
 in the same way as $T$ does. 
 Hence, the $L^{3p}(\gO;C^{n,\gg}(\IS^2))$-norm of $1/a$ is finite.
 The induction hypothesis, Theorem~\ref{thm:logNiGRFsProp}, and Corollary~\ref{cor:apint} 
 imply that the right hand side of the previous inequality
 is finite.
 This completes the induction.
 We conclude that the $L^p(\gO;H^{1+s}(\IS^2))$-norm of $u$ is finite. 
 The proof for $u^L$, $L\in\N_0$, is analogous. 
 The uniform boundedness of the $L^p(\gO;H^{1+s}(\IS^2))$-norm of $u^L$ in $L\in\N_0$ is implied
 by Theorem~\ref{thm:logNiGRFsProp} and Corollary~\ref{cor:apint}.
\end{proof}
%
\section{Discretization}\label{sec:Discr}
\subsection{Finite Element Methods}
\label{sec:FEMS2}

In Proposition~\ref{prop:EllSPDE_approx} we analyzed the error 
that occurs when we consider the solution $u^L=\Phi_{f}(a^L)$ 
to the SPDE~\eqref{eq:ellSPDEweak} with respect to the 
approximate isotropic lognormal RF $a^L=\exp(\Pi_L T)$ for $L\in\N_0$,
where $a^L=\exp(\Pi_L T)$ can be simulated via 
the truncated \KL expansion of the iGRF $T$ for every $L\in\N_0$.
In this section we aim at a spatial discretization to numerically 
simulate realizations of $u^L$, $L\in\N_0$,
with a \emph{Galerkin} Finite Element Method and analyze 
the error in the $L^p(\Omega;H^1(\IS^2)/\R)$-norm for $p\in[1,+\infty)$.

We review basic results on the deterministic theory of FEs on $\IS^2$ 
as required in the ensuing analysis.
FEs on surfaces to approximate solutions of elliptic PDEs
appear to have been first introduced in~\cite{Dziuk88}.
There, first order convergence estimates are obtained 
using affine approximations of the surface. 
Higher order estimates are shown in~\cite{Demlow2009}, 
where also an FE Method is defined on the surface 
so as to avoid a surface approximation error. 
We refer to~\cite[Section~2.6]{Demlow2009} for details.

Given a regular, quasiuniform triangulation 
$\cT$ of $\IS^2$ into parametric, curvilinear triangles $K\in \cT$
of mesh width $h > 0$ (which we indicate by tagging $\cT$ with 
the subscript $h$, i.e., by writing $\cT_h$),
we define $S^{k}(\IS^2,\cT_h)$ to be the space of continuous,
piecewise parametric polynomials of degree $k\geq 1$ on 
the triangulation $\cT_h$ of $\IS^2$ 
and equip it with the $H^1(\IS^2)$-norm. 
To approximate functions in $H^1(\IS^2)/\R$ 
we define the subspace of $S^{k}(\IS^2,\cT_h)$ of functions 
that have zero average, 
i.e., 
\begin{equation*}
 V^{h,k} := \{ v^h \in S^{k}(\IS^2,\cT_h), (v^h,1)=0\}
 .
\end{equation*}
Then, $V^{h,k}\subset H^1(\IS^2)/\R$ and we equip it with
the $H^1(\IS^2)/\R$-norm. The FE spaces $S^{k}(\IS^2,\cT_h)$ and $V^{h,k}$, $h>0$, 
are of finite dimension such that ${\rm dim}(S^{k}(\IS^2,\cT_h)) = {\rm dim}(V^{h,k})+1$.
Also it holds that the degrees of freedom 
$N_h := {\rm dim}(V^{h,k}) = \Op(h^{-2})$ as $h\rightarrow 0$ 
for fixed polynomial degree $k\in\N$.
We refer to \cite[Chapter~4]{SaScBEM11} for details and remark that 
we will tag elements of $V^{h,k}$ respectively $S^{k}(\IS^2,\cT_h)$ 
only with the mesh width $h$ keeping in mind that they implicitly
also depend on the polynomial degree $k$ 
of the FE space, i.e., let $v^h \in V^{h,k}$.

For every $\tilde{a}\in C^0_+(\IS^2)$, $h>0$, and $k\in\N$,
we consider the variational formulation of the deterministic, 
elliptic PDE~\eqref{eq:ellPDEweak} 
over the finite dimensional space $V^{h,k}$: 
find a \emph{Galerkin FE solution} 
$\tilde{u}^h\in V^{h,k}$ such that
\begin{equation}\label{eq:ellPDEGalerkin}
 (\tilde{a}\nabla_{\IS^2}\tilde{u}^h, \nabla_{\IS^2} v^h) 
 =
 f(v^h)
 \quad
 \forall v^h \in V^{h,k}
 .
\end{equation}
The conformity of the FE Method, i.e., $V^{h,k}\subset H^1(\IS^2)/\R$, 
implies with~\eqref{eq:bilinearform_continuous} and~\eqref{eq:bilinearform_coercive} 
that the bilinear form $(\tilde{a}\nabla_{\IS^2}\cdot, \nabla_{\IS^2} \cdot)$ on $V^{h,k}\times V^{h,k}$
is continuous and coercive with coercivity constant $(\min_{x\in\IS^2}\tilde{a}(x))^{-1}$ 
which is independent of $h$ and of $k$.

Hence, by the Lax--Milgram lemma, 
the Galerkin approximation $\tilde{u}^h \in V^{h,k}$ 
exists and is the unique solution of~\eqref{eq:ellPDEGalerkin}.
Also $\tilde{u}^h$ satisfies the estimate in~\eqref{est:Pth_H1} uniformly in $h>0$, i.e.,
\begin{equation}\label{est:Pth_H1_FE}
 \|\tilde{u}^h\|_{H^1(\IS^2)/\R}
 \leq 
 \frac{1}{\min_{x\in\IS^2}\tilde{a}(x)}\sqrt{\frac{3}{2}}\|f\|_{H^{-1}(\IS^2)}
 .
\end{equation}
As in the previous section we introduce a solution mapping $\Phi_{f}^{h,k}$ 
that maps the coefficient $\tilde{a}\in  C^0_+(\IS^2)$ 
to the respective unique Galerkin FE solution $\tilde{u}^h\in V^{h,k}$ 
by
\begin{equation*}
\Phi_{f}^{h,k}:  C^0_+(\IS^2) \rightarrow V^{h,k}
.
\end{equation*}
Continuity follows as in Proposition~\ref{prop:solution_map_cont} with Lemma~\ref{lemma:Strang} and is stated in the following proposition.
\begin{proposition}\label{prop:solution_map_cont_FE}
 $\Phi_{f}^{h,k} : C^0_+(\IS^2) \rightarrow V^{h,k}$
 is continuous for every $h>0$ and $k\in\N$.
\end{proposition}
Functions in $H^{1+s}(\IS^2)$ 
and in particular solutions to~\eqref{eq:ellPDEweak}
can be approximated in $S^{k}(\IS^2,\cT_h)$, $s,h>0$, and $k\in\N$, cp.\ 
\cite[Proposition~2.7]{Demlow2009}.
We will phrase this in terms of the solution mappings 
$\Phi_{f}$ and $\Phi_{f}^{h,k}$, $h>0$ and $k\in\N$, in the following proposition.
The proof uses this well-known approximation property
of $S^{k}(\IS^2,\cT_h)$, $h>0$ and $k\in\N$,
in $H^1(\IS^2)$ in combination with C\'ea's lemma. 
For details, we refer the reader to Appendix~\ref{app:proof:prop:solution_map_approx_FE}.
\begin{proposition}\label{prop:solution_map_approx_FE}
 Let $k\in\N$ be the polynomial degree of the FE spaces $V^{h,k}$, $h>0$, 
 and let $\iota\in\N_0$ and $\gg\in(0,1)$.
 For every $s\in(0,\iota+\gg)$ such that  $f\in H^{-1+s}(\IS^2)$,
 there exists a constant $C_s$ such that
 for every $h>0$ and every $\tilde{a}\in C^{\iota,\gg}(\IS^2)\cap C^0_+(\IS^2)$,
 it holds that
 \begin{equation*}
 \|\Phi_{f}(\tilde{a}) - \Phi_{f}^{h,k}(\tilde{a})\|_{H^1(\IS^2)/\R}
 \leq
 C_s \; \frac{\|\tilde{a}\|_{C^0(\IS^2)}}{\min_{x\in\IS^2}\tilde{a}(x)} 
\|\Phi_{f}(\tilde{a})\|_{H^{1+s}(\IS^2)} \; h^{\min\{s,k\}}
.
 \end{equation*}
\end{proposition}
Since the mappings $\Phi_{f}^{h,k}$, $h>0,k\in\N$, 
are continuous due to Proposition~\ref{prop:solution_map_cont_FE}, 
the introduced theory on Galerkin FE Methods is applicable to our stochastic framework. 
Indeed, for every $L\in\N_0$, $h>0$, and $k\in\N$, 
the problem to find a strongly $H^1(\IS^2)/\R$-measurable $u^{L,h}$ such that
\begin{equation}\label{eq:ellSPDEweak_FE}
 (a^L \nabla_{\IS^2} u^{L,h}, \nabla_{\IS^2} v^h) 
 = 
 f(v^h)
 \quad \forall v^h\in V^{h,k}
\end{equation}
admits a unique solution 
by setting $u^{L,h}:=\Phi_{f}^{h,k}(a^L)$, 
where we omit $k$ in our notation of the solution.
The strong $H^1(\IS^2)/\R$-measurability of $u^{L,h}$ follows 
from the strong $C^0(\IS^2)$-measurability
of $a^L$ and the continuity of $\Phi_{f}^{h,k}$ with Lemma~\ref{lemma:comp_meas_maps}.
Moreover Corollary~\ref{cor:apint} implies with~\eqref{est:Pth_H1_FE}
that for every $p\in[1,+\infty)$, there exists a constant $C_p$ such that
for every $L\in\N_0$ and every $h>0$, it holds that
\begin{equation}\label{est:Lp_H1_FE}
 \|u^{L,h}\|_{L^p(\gO;H^1(\IS^2)/\R)}
 \leq
 \|1/\check{a}^L\|_{L^p(\gO)}
 \sqrt{\frac{3}{2}}\|f\|_{H^{-1}(\IS^2)}
 \leq C_p 
 .
\end{equation}

With the given properties of the Galerkin Finite Elements, we are now able to prove the extension of Proposition~\ref{prop:EllSPDE_approx} to space discretizations.
\begin{theorem}\label{thm:trunc_hFEM_Conv}
Let the assumptions of Theorem~\ref{thm:PthReg_Sobolev} be satisfied.
Let $u=\Phi_{f}(a)$ be the unique solution of~\eqref{eq:ellSPDEweak} and for every $h>0$, 
let $u^{L,h}=\Phi_{f}^{h,k}(a^L)$ be the unique Galerkin 
FE solution of~\eqref{eq:ellSPDEweak_FE} for $k\in\N$.
Then, for every $s\in(0,\gb/2)$ such that  $f\in H^{-1+s}(\IS^2)$
and every $p\in[1,+\infty)$, there exists a constant $C_{p,s}$ 
such that for every $h>0$ and every $L\in\N_0$, it holds that
\begin{equation*}
\| u  - u^{L,h} \|_{L^p(\Omega;H^1(\IS^2)/\R)} 
\leq 
C_{p,s}
(L^{-s} + h^{\min\{s,k\}})
.
\end{equation*}
\end{theorem}
\begin{proof}
 Let us set $u^L:=\Phi_{f}(a^L)$ for every $L\in\N_0$.
 A twofold 
 application of H\"older's inequality implies
 with Proposition~\ref{prop:solution_map_approx_FE}
 that there exists a constant $C_s$
 such that for every $L\in\N_0$ and every $h>0$, it holds that
 \begin{equation*}
  \|u^L - u^{L,h}\|_{L^p(\gO;H^1(\IS^2)/\R)}
  \leq 
  C_s \|a^L\|_{L^{3p}(\gO;C^0(\IS^2))} \|1/\check{a}^L\|_{L^{3p}(\gO)} \|u^L\|_{L^{3p}(\gO;H^{1+s}(\IS^2))}
  \; h^{\min\{s,k\}}
  .
 \end{equation*}
 Due to Theorem~\ref{thm:logNiGRFsProp}, Corollary~\ref{cor:apint}, and Theorem~\ref{thm:PthReg_Sobolev}
 there exists a constant $\hat{C}_{p,s}$ such that for every $L\in\N_0$, it holds that
 \begin{equation*}
 C_s\|a^L\|_{L^{3p}(\gO;C^0(\IS^2))} \|1/\check{a}^L\|_{L^{3p}(\gO)} \|u^L\|_{L^{3p}(\gO;H^{1+s}(\IS^2))}\leq C_{p,s}.
 \end{equation*}
 Let $\varepsilon:=\gb-2s\in(0,\gb)$. 
We apply the triangle inequality and conclude with Proposition~\ref{prop:EllSPDE_approx} that
there exists a constant that we also denote by $\hat{C}_{p,s}$ such that for every $L\in\N_0$
and every $h>0$, it holds that
\begin{equation*}
 \|u-u^{L,h}\|_{L^p(\Omega;H^1(\IS^2))}
 \leq \hat{C}_{p,s}\Bigl ( \sum_{\ell>L} A_\ell \ell^{1+\varepsilon} \Bigr)^{\frac{1}{2}}
 +\hat{C}_{p,s}\; h^{\min\{s,k\}}.
\end{equation*}
We further bound
\begin{equation*}
 \sum_{\ell>L} A_\ell \ell^{1+\varepsilon}
 \leq (L^{-1})^{\gb-\varepsilon}\sum_{\ell>L} A_\ell \ell^{1+\gb}
 \leq (L^{-1})^{2s}\sum_{\ell\geq0} A_\ell \ell^{1+\gb}.
\end{equation*}
Since $\sum_{\ell\geq0} A_\ell \ell^{1+\gb}<+\infty$ by assumption, we conclude the proof of the theorem.
\end{proof}
%

%
\subsection{Spectral Methods}
\label{sec:SpecElS2}
In Theorem~\ref{thm:trunc_hFEM_Conv} we established a rate of convergence
for Galerkin approximations of the stochastic solution 
in subspaces $V^{h,k}$ of continuous, piecewise 
polynomial functions on a quasiuniform triangulation $\cT_h$ on $\IS^2$.
The obtained bound for the convergence rate in Theorem~\ref{thm:trunc_hFEM_Conv} indicated 
an asymptotic convergence order $N_h^{-\min\{s,k\}/2}$ as 
$N_h = {\rm dim}(V^{h,k}) \rightarrow +\infty$, i.e., the convergence rate
is limited by the regularity of the solutions (as expressed in the 
Sobolev scale parameter $s \geq 0$) and by the polynomial degree 
$k\in\N$ of the Finite Elements used in the discretization. 
If, in particular, the Sobolev regularity of the solution is high, 
i.e., if $s > 0$ is large, the convergence of the 
Galerkin FE approximations $u^{L,h}$ defined in~\eqref{eq:ellSPDEweak_FE} is limited 
by the order $k$ of the used Finite Elements. 
Spectral Elements do not have this drawback.

To introduce them, we recall the space ${\cH_{0:L^u}}\subset H^1(\IS^2)$ 
spanned by spherical harmonics of order at most $L^u$ defined in~\eqref{eq:YL}.
Since we are interested in a conforming method, 
we restrict ourselves to the functions that are orthogonal to constants
as in the FE case, i.e., we consider $\cH_{1:L^u}$ as Spectral Element spaces, $L^u\in\N$.
In the following 
the index $L^a$ refers to the degree of the approximation of 
$a$ and $L^u$ refers to the degree of the Spectral Element space.
Its dimension is $N_{L^{u}} := {\rm dim}({\cH_{1:L^u}}) = \Op( (L^{u})^2 )$ as $L^{u}\rightarrow +\infty$, 
and is also referred to as degrees of freedom.
Let $a=\exp(T)$ be an isotropic lognormal RF that results from a continuous iGRF $T$
satisfying~\eqref{eq:AngSpecDec} for some $\beta>0$.
Similarly to~\eqref{eq:ellSPDEweak_FE}, for every $L^a,L^u\in\N_0$, 
we define a Galerkin approximation as the solution of the problem to
find a strongly $H^1(\IS^2)/\R$-measurable $u^{L^a,L^{u}}$ that takes values in $\cH_{1:L^u}$
such that
\begin{equation*}
(a^{L^a} \nabla_{\IS^2} u^{L^a,L^{u}}, \nabla_{\IS^2} v^{L^{u}})  
= 
f(v^{L^{u}})
\quad \forall v^{L^{u}} \in \cH_{1:L^u}
.
\end{equation*}
The coercivity
of the bilinear form $(a^{L^a} \nabla_{\IS^2} \cdot, \nabla_{\IS^2} \cdot)$ implies that
$u^{L^a,L^{u}}$ exists and is unique,
since $\cH_{1:L^u} \subset H^1(\IS^2)/\R$ is a closed subspace.
Strong $H^1(\IS^2)/\R$-measurability of $u^{L^a,L^{u}}$ 
follows in the same way as in Section~\ref{sec:FEMS2}.

Let us conclude this subsection with the spectral version of Theorem~\ref{thm:trunc_hFEM_Conv} which expresses the convergence rate just in terms of the Sobolev regularity of the solution of the original problem.
\begin{theorem}\label{thm:ellSPDEN}
Let the assumptions of Theorem~\ref{thm:PthReg_Sobolev} be satisfied.
For every $s\in(0,\gb/2)$ such that $f\in H^{-1+s}(\IS^2)$ and 
for every $p\in[1,+\infty)$,
there exists a constant $C_{p,s}$ such that for every $L^a,L^u\in\N_0$, it holds that
\begin{equation*}
\| u  - u^{L^a,L^u} \|_{L^p(\Omega;H^1(\IS^2)/\R)}
\leq
C_{p,s} \Bigl( (L^a)^{-s} + (L^u)^{-s}\Bigr).
\end{equation*}
\end{theorem}
\begin{proof} The proof is similar to that of Theorem~\ref{thm:trunc_hFEM_Conv}.
We use the approximation result Proposition~\ref{prop:EllSPDE_approx}, the quasioptimality, 
the regularity result 
Theorem~\ref{thm:PthReg_Sobolev}, and the approximation property of 
$\cH_{0:L^u}$ in Proposition~\ref{prop:FAppr} to conclude the assertion.
\end{proof}
\section{MLMC convergence analysis}

\label{sec:MLMC}
In this section we aim at approximating the expectation of the solution of~\eqref{eq:ellSPDEweak} $\E(u)$.
So far we established for FE approximations of~$u$ in Section~\ref{sec:FEMS2} 
that the constructed double indexed sequence
$(u^{L,h},L\in\N_0, h>0)$ converges to $u$ in $L^p(\gO;H^1(\IS^2)/\R)$ for every $p\in[1,+\infty)$ with 
a particular convergence rate, cp.\ Theorem~\ref{thm:trunc_hFEM_Conv}. 
The remaining part of the numerical analysis is to approximate 
$\E(u^{L,h})$ for $L\in\N_0$ and $h>0$ with a sampling method.
To this end, we apply an MLMC estimator in order to reduce the computational cost that a 
conventional Monte Carlo simulation would incur. 

The error analysis of MLMC discretizations is standard, by now, and 
our development is analogous to
those carried out in~\cite{GKSS_MLMC13, BSZ11, BL12_2}. 
In particular, in~\cite{GKSS_MLMC13} the error from truncating a \KL expansion
of the Gaussian random field was considered. 
In contrast to the situation there, 
we will benefit in our analysis from the knowledge of the properties
of iGRFs and of the behavior of their \KL expansions that we 
developed in Section~\ref{sec:iGRF}. 
This relieves us from additional assumptions on the \KL eigenfunctions, 
on the behavior of the truncated
\KL expansion, and on the iGRF itself,
apart from summability assumptions on the angular power spectrum.

We introduce the usual Monte Carlo (MC) estimator and the MLMC estimator in a general setting.
Let $(V,\|\cdot\|_V)$ be a separable Hilbert space.
For every $v\in L^2(\gO;V)$, let  
$(\hat{v}_i,i\in\N)$ be a sequence of independent, identically distributed 
random variables in~$L^2(\gO;V)$ 
such that they are independent from~$v$ and have the same law as $v$.
For every $M\in\N$, the MC estimator~$E_M(v)$ of~$v$ is then defined by
\begin{equation*}
 E_M ( v ) := \frac{1}{M} \sum_{i=1}^M \hat{v}_i
 .
\end{equation*}
It is well-known that for every $v\in L^2(\gO;V)$ and every $M\in\N$, it holds that
\begin{equation}\label{eq:MC_estimator_var}
 \|\E(v) - E_M(v)\|^2_{L^2(\gO;V)}
 =
 \frac{1}{M} \| v - \E(v)\|^2_{L^2(\gO;V)}
 =
  \frac{1}{M} ( \| v \|^2_{L^2(\gO;V)} - \|\E(v)\|^2_{V} )
 .
\end{equation}
For every $L^2(\gO;V)$-valued sequence $(v^j,j\in\N_0)$, 
we consider a finite telescoping sum expansion with the convention that $v^{-1}=0$,
i.e., for every $J'\in\N_0$, it holds that
\begin{equation*}
 v^{J'}=\sum_{j=0}^{J'} v^j - v^{j-1}
 ,
 \end{equation*}
and define for every $\N$-valued sequence $(M_j,j=0,\dots,J)$, $J\in\N_0$, 
the MLMC estimator $E^J$ of $v^J$ by
\begin{equation}\label{eq:MLMC_est_def}
 E^J(v^J) := \sum_{j=0}^J E_{M_j}(v^{j} - v^{j-1})
\end{equation}
such that the MC estimators $(E_{M_j}(v^{j} - v^{j-1}),j=0,\dots,J)$ are independent.

In the following lemma we express the error introduced by the MLMC estimator in terms of the errors of the approximations and the numbers of samples chosen on each level.
\begin{lemma}\label{lemma:MLMC_general_estimate}
 For every $L^2(\gO;V)$-valued sequence $(v^j,j\in\N_0)$ and every integer-valued $(M_j,j=0,\dots,J)$ sequence with finite $J\in\N_0$, the MLMC estimator $E^J(v^J)$ satisfies that
\begin{equation*}
\|\E(v^J) - E^J(v^J)\|^2_{L^2(\gO;V)}
=
\sum_{j=0}^J\frac{1}{M_j}\Bigl(\| v^j - v^{j-1}\|^2_{L^2(\gO;V)} - \|\E(v^j - v^{j-1})\|^2_V\Bigr)
.
\end{equation*}
\end{lemma}
\begin{proof}
 The independence of the MC estimators in~\eqref{eq:MLMC_est_def}
 on the different levels $(M_j,j=0,\dots,J)$ 
and~\eqref{eq:MC_estimator_var} imply that
 \begin{align*}
  \|\E(v^J) - E^J(v^J)\|^2_{L^2(\gO;V)}
  &=
  \|\sum_{j=0}^J \E(v^j - v^{j-1}) - E_{M_j}(v^j - v^{j-1})\|^2_{L^2(\gO;V)}\\
  &=
  \sum_{j=0}^J\|\E(v^j - v^{j-1}) - E_{M_j}(v^j - v^{j-1})\|^2_{L^2(\gO;V)}\\ 
  &=
   \sum_{j=0}^J\frac{1}{M_j}(\| v^j - v^{j-1}\|^2_{L^2(\gO;V)} - \|\E(v^j - v^{j-1})\|^2_V)
   .
   \qedhere
 \end{align*}
\end{proof}
%
%
%
%
After having computed the error introduced by an MLMC estimator, 
we are now in state to compute the overall error of the full discretization in terms of the regularity of the solution, the approximation of the iGRF, the FE discretization, and the sample sizes.
\begin{theorem}\label{thm:MLMC_error_est_u^L_h}
 Let the assumptions of Theorem~\ref{thm:trunc_hFEM_Conv} be satisfied
 and let $u$ be the unique solution to~\eqref{eq:ellSPDEweak}.
 Consider for any increasing $\N$-valued sequence $(L_j,j\in\N_0)$ and
 decreasing positive sequence $(h_j,j\in\N_0)$ the corresponding sequence of FE solutions $(u^{L_j,h_j},j\in\N_0)$ to~\eqref{eq:ellSPDEweak_FE},
 i.e., for fixed $k\in\N$ and for every $j\in\N_0$, $u^{L_j,h_j}$ satisfies
 \begin{equation*}
  (a^{L_j}\nabla_{\IS^2} u^{L_j,h_j},\nabla_{\IS^2}v^{h_j}) 
  = 
  f(v^{h_j})
  \qquad 
  \forall v^{h_j} \in V^{h_j,k}
  .
 \end{equation*}
 Then, for every $s\in(0,\gb/2)$, 
 there exists a constant $C_s$ such that for every $\N$-valued sequence $(M_j,j=0.\dots,J)$, $J\in\N_0$, it holds that 
 \begin{equation*}
  \|\E(u) - E^J(u^{L_J,h_J})\|_{L^2(\gO;H^1(\IS^2)/\R)}
  \leq C_s\Bigl( \frac{1}{M_0} + \sum_{j=1}^{J} \frac{L_{j-1}^{-2s} + h_{j-1}^{2\min\{s,k\}}}{M_{j}} + L_J^{-2s} + h_J^{2\min\{s,k\}} \Bigr)^{1/2}
  .
 \end{equation*}
\end{theorem}
\begin{proof}
Theorem~\ref{thm:trunc_hFEM_Conv} implies 
that there exists a constant $\hat{C}_s$ independent
of $(L_j,j\in\N_0)$, $(h_j,j\in\N_0)$, and $J$ such that 
for every $j=1,\dots,J$, it holds that
\begin{align*}
&\| u^{L_j,h_j} - u^{L_{j-1},h_{j-1}}\|_{L^2(\gO;H^1(\IS^2)/\R)} \\
&\qquad\leq
\| u - u^{L_{j},h_{j}}\|_{L^2(\gO;H^1(\IS^2)/\R)} 
+
\| u - u^{L_{j-1},h_{j-1}}\|_{L^2(\gO;H^1(\IS^2)/\R)}\\
&\qquad\leq 
\hat{C}_s(L_{j}^{-s} + h_{j}^{\min\{s,k\}} + L_{j-1}^{-s} + h_{j-1}^{\min\{s,k\}})
\leq 2\hat{C}_s(L_{j-1}^{-s} + h_{j-1}^{\min\{s,k\}} )
,
\end{align*}
where we apply that $(L_j,j\in\N_0)$ is increasing and 
 $(h_j,j\in\N_0)$ is decreasing and recall 
that the elements of $(u^{L_{j},h_{j}}$, $j\in\N_0)$ depend on the polynomial
 degree $k$ of $V^{h,k}$.
 Another implication of Theorem~\ref{thm:trunc_hFEM_Conv}
 is that for the same constant $\hat{C}_s$, it holds that
\begin{equation*}
\|\E(u) - \E(u^{L_J,h_J})\|_{H^1(\IS^2)/\R}
\leq \|u - u^{{L_J,h_J}}\|_{L^1(\gO;H^1(\IS^2))}
\leq \hat{C}_s(L_J^{-s} + h_J^{\min\{s,k\}}),
\end{equation*}
and due to~\eqref{est:Lp_H1_FE} there exists a constant $\hat{C}$ that is independent of $L_0$ and $h_0$ such that
\begin{equation*}
\| u^{L_0,h_0}\|_{L^2(\gO;H^1(\IS^2)/\R)}
\leq \hat{C}
.
\end{equation*}
Hence, we conclude the claim of this theorem with the triangle inequality,
 Lemma~\ref{lemma:MLMC_general_estimate}, 
and with the elementary inequality that 
$(r_1+r_2)^2\leq 2(r_1^2 +r_2^2)$ for every $r_1,r_2\in\R$.
Specifically, for $C_s:=4\max\{\hat{C}_s,\hat{C}\}$, it holds that
\begin{align*}
&\|\E(u) - E^J(u^{L_J,h_J})\|_{L^2(\gO;H^1(\IS^2)/\R)} \\
&\qquad\leq 
\|\E(u) - \E(u^{L_J,h_J})\|_{H^1(\IS^2)/\R} + \|\E(u^{L_J,h_J}) - E^J(u^{L_J,h_J})\|_{L^2(\gO;H^1(\IS^2)/\R)}\\
&\qquad\leq \sqrt{2}\Bigl( 
\hat{C}^2_s  (L_J^{-s} + h_J^{\min\{s,k\}})^2 + \frac{\hat{C}^2}{M_0} + 4\hat{C}_s^2\sum_{j=1}^J \frac{(L_{j-1}^{-s} + h_{j-1}^{\min\{s,k\}} )^2}{M_j}
\Bigr)^{1/2}\\
&\qquad\leq
C_s\Bigl( 
   \frac{1}{M_0} + \sum_{j=1}^J \frac{(L_{j-1}^{-2s} + h_{j-1}^{2 \min\{s,k\}} )}{M_j} + L_J^{-2s} + h_J^{2 \min\{s,k\}}
\Bigr)^{1/2}
.
\qedhere
\end{align*}
\end{proof}
We remark that Theorem~\ref{thm:MLMC_error_est_u^L_h} also covers the convergence analysis of the usual Monte Carlo estimator by the choice $J=0$.

It is natural to require
\begin{equation}\label{eq:MLMC_h_j}
 h_j  = \Op( 2^{-j} h_0),
  \qquad j\in\N_0,
\end{equation}
for some initial mesh width $h_0>0$.
Generally,
 one attempts to equilibrate the error contributions of the approximations of the noise, in space, and of the expectation. 
From Theorem~\ref{thm:trunc_hFEM_Conv} or Theorem~\ref{thm:MLMC_error_est_u^L_h}
we see that to equilibrate the error contributions from the 
truncation of the \KL expansion of the continuous iGRF and the error
contribution from the Galerkin FE approximation
we need to choose the increasing sequence $(L_j,j\in\N_0)$ 
\emph{comparably} to $(h_j^{-1},j\in\N_0)$, i.e., 
there exists a constant $C$ with $C^{-1} h_j\leq (L_j)^{-1}\leq C h_j$ for every $j\in\N_0$.
Hence, we consider
\begin{equation}\label{eq:MLMC_L_j}
 L_j  :=\left\lceil\frac{h_0}{h_j} \right\rceil L_0,
  \qquad j\in\N_0,
\end{equation}
for some initial truncation level $L_0\in\N$.
Under our only assumption that the angular power spectrum of the continuous iGRF
satisfies~\eqref{eq:AngSpecDec} for some $\gb>0$ we obtained with Theorem~\ref{thm:PthReg_Sobolev}
that the unique solution $u$ to~\eqref{eq:ellSPDEweak} 
is in $L^p(\gO;H^{1+s}(\IS^2))$ for every $s\in[0,\gb/2)$ and every $p\in[1,+\infty)$.
To determine  the sample sizes
for a given $\beta>0$, we fix $s\in(0,\beta/2)$ such that $s\leq k$, where $k$ denotes 
the polynomial degree of the FE space.
A possible choice of the sample numbers $(M_j,j=0,\dots,J)$, $J\in\N_0$, 
in the MLMC estimator is to equilibrate the error contributions of the MLMC
estimator across the discretization levels according to Theorem~\ref{thm:MLMC_error_est_u^L_h}.
This leads to the following choice: 
for a given maximal discretization level $J\in\N_0$, 
we set
\begin{equation}\label{eq:MLMC_levels}
 M_0= \lceil h_J^{-2s} \kappa\rceil  
 \quad \text{and}\quad 
 M_j = \Bigl\lceil\Bigl(\frac{h_{j-1}}{h_J}\Bigr)^{2s}j^{1+\varepsilon}\kappa\Bigr\rceil   
\end{equation}
for $j=1,\dots,J$, 
a scaling factor $\kappa\geq  2^{-2s}$ (allow $\kappa>0$ if $J=0$), and 
a positive constant $\varepsilon>0$.
If $s>k$ we make the same choices as in~\eqref{eq:MLMC_levels} 
with $s$ replaced by $k$.
\begin{corollary}\label{cor:err_and_work_bounds}
 Let $J\in\N_0$ and $\varepsilon>0$ be fixed. Assume that the conditions of Theorem~\ref{thm:MLMC_error_est_u^L_h} 
 are satisfied for some $\beta>0$ and choose $(L_j,h_j,M_j,j=0,\dots,J)$ 
 according to~\eqref{eq:MLMC_L_j}, \eqref{eq:MLMC_h_j}, and~\eqref{eq:MLMC_levels}.
 Denote by~$\zeta$ the Riemann zeta function.
 Then, for every $s\in(0,\beta/2)$, 
 there exists $C_s>0$ such that
 \begin{equation*}
 \|\E(u) - E^J(u^{L_J,h_J})\|_{L^2(\gO;H^1(\IS^2)/\R)}
 \leq C_s{\Bigl( \zeta(1+\varepsilon)\frac{1}{\kappa} + 1 \Bigr)^{1/2}}\;  h_J^{\min\{s,k\}}
 .
\end{equation*}
If for $\eta_1>0$ and $\eta_2\geq 0$, 
the work to compute one sample of $u^{L_j,h_j}$ is comparable to
$h_j^{-2\eta_1}\log^{\eta_2}(h_j^{-2})$, $j=0,\dots,J$,
then the total work to compute $E^J(u^{L_J,h_J})$ satisfies
\begin{equation*}
\cW_J 
= 
\begin{cases}
\Op(h_J^{-2\min\{s,k\}}\gk) 
= 
\Op(2^{2\min\{s,k\}J}\gk)
& \min\{s,k\} > \eta_1 \\
\Op(h_J^{-2\eta_1}\max\{J,1\}^{\eta_2+2+\varepsilon}\gk) 
= 
\Op(2^{2J\eta_1}\max\{J,1\}^{\eta_2 + 2 +\varepsilon}\gk)
& \min\{s,k\} \leq \eta_1 
\end{cases},
\end{equation*}
where the contributions of $h_0^{-1},L_0, {\eta_1, \eta_2}$ are absorbed into 
the Landau symbols.
\end{corollary}
\begin{proof}
Let $s_0:=\min\{s,k\}$.
The error estimate follows from the choices of the values for $(L_j,h_j,M_j,j=0,\dots,J)$ 
by Theorem~\ref{thm:MLMC_error_est_u^L_h},
i.e., we 
conclude that
\begin{align*}
&\|\E(u) - E^J(u^{L_J,h_J})\|_{L^2(\gO;H^1(\IS^2)/\R)}\\
&\qquad\leq 
 \hat{C}_s \Bigl( \frac{1}{M_0} + \sum_{j=1}^J \frac{L_{j-1}^{-2s} + h_{j-1}^{2s_0}}{M_j} + L_J^{-2s} + h_J^{2s_0}\Bigr)^{1/2}
\\
 &\qquad\leq \hat{C}_s {\Bigl( \frac{1}{\kappa} +  \frac{1}{\kappa}\Big(\frac{1}{(L_0h_0)^{2s_0}}+1\Big)\zeta(1+\varepsilon) + \frac{1}{(L_0h_0)^{2s_0}}+1  \Bigr)^{1/2}} h_J^{s_0}
 ,
\end{align*}
where $\hat{C}_s$ is the constant from Theorem~\ref{thm:MLMC_error_est_u^L_h}.
Since $\zeta(1+\varepsilon)>1$ for every $\varepsilon>0$, 
we obtain the claimed estimate with $C_s := \hat{C}_s\sqrt{(L_0 h_0)^{-2s_0} +2}$.
To prove the bound on the computational work, 
we insert the values for $M_j$ and $h_j$ and obtain
\begin{align*} 
 \cW_J 
 &\leq C_1 \Bigl(
 M_0 h_0^{-2\eta_1}\log^{\eta_2}(h_0^{-2})  + \sum_{j=1}^J M_j( h_j^{-2\eta_1}\log^{\eta_2}(h_j^{-2}) 
 + h_{j-1}^{-2\eta_1}\log^{\eta_2}(h_{j-1}^{-2}))
 \Bigr)
 \\
 &
  \leq C_2 \gk \Bigl(
 2^{2s_0 J} + \sum_{j=1}^J 2^{2s_0(J-j+1) +2j\eta_1} j^{\eta_2+1+\varepsilon}
 \Bigr)
 ,
\end{align*}
where the constant $C_2>0$ depends on $C_1>0$, $L_0$, $h_0$, $\eta_1$, and $\eta_2$.
If $s_0\leq \eta_1$, then $\cW_J = \Op(2^{2J\eta_1}J^{\eta_2+2+\varepsilon}\gk) = \Op(h_J^{-2\eta_1}J^{\eta_2+2+\varepsilon}\gk)$.
In the other case that $s_0>\eta_1$, it follows with the fact 
that $\sum_{j\geq 1} \rho^j j^{\eta_2+1+\varepsilon} <+\infty$
for every $\rho\in(0,1)$ that
\begin{align*}
 \cW_J 
 &\leq C_2
 2^{2s_0J}\gk \Bigl(1 + \sum_{j=1}^J 2^{-2j(s_0-\eta_1)} j^{\eta_2+1+\varepsilon}\Bigr)
 =
 \Op(2^{2s_0 J}\gk)
 =
 \Op(h_J^{-2s_0}\gk)
 ,
\end{align*}
which finishes the proof of the corollary.
\end{proof}%
Note that the choices $(M_j,j=0,\dots,J)$, $J\in\N_0$ in~\eqref{eq:MLMC_levels}
depend on the regularity of the solution~$u$ of~\eqref{eq:ellSPDEweak}.
However, the closer $s$ is to $\gb/2$
the harder it should be to observe the convergence behavior that is theoretically 
guaranteed by Theorem~\ref{thm:MLMC_error_est_u^L_h}, 
because constants may become arbitrarily large.
We conclude the theoretical part of 
the paper with several remarks on the convergence bounds.

\begin{remark}\label{rmk:MLMC_spect_meth}
The proof of Theorem~\ref{thm:MLMC_error_est_u^L_h} is not restricted to the considered FE Methods above.
If the conditions of Theorem~\ref{thm:ellSPDEN} are satisfied with $\gb>0$, 
an analogous argument implies the respective statement in the case of Spectral Methods,
i.e., for every $s\in(0,\gb/2)$, there exists a
constant $C_s>0$ such that  for $J\in\N_0$,
\begin{equation*}
 \|\E(u) - E^J(u^{L^a_J,L^u_J})\|_{L^2(\gO;H^1(\IS^2)/\R)}
  \leq C_s\Bigl( \frac{1}{M_0} + \sum_{j=1}^{J} \frac{1}{M_{j}}L_{j-1}^{-2s}  + L_J^{-2s} \Bigr)^{1/2}
  ,
\end{equation*}
where the degrees of $a^{L_j^a}$ and of $\cH_{1:L_j^u}$, $j\in\N_0$,
are chosen as increasing sequences that define $L_j:= \min\{L_j^a, L_j^u\}$, $j\in\N_0$.
As in the FE case the number of samples to equilibrate the MC errors on the levels 
can be chosen $M_0 :=\lceil L_J^{2s}\gk\rceil$ 
and $M_j := \lceil(L_{J}/L_{j-1})^{2s}j^{1+\varepsilon}\gk\rceil$, $j=1,\dots,J$, for a positive constant $\varepsilon>0$ 
and $\gk\geq (L_{J}/L_{J-1})^{-2s}$ (allow $\gk>0$ if $J=0$). 
Hence, there exists $C_s>0$ such that
\begin{equation*}
 \|\E(u) - E^J(u^{L^a_J,L^u_J})\|_{L^2(\gO;H^1(\IS^2)/\R)}
 \leq C_s\Bigl(  \zeta(1+\varepsilon)\frac{1}{\kappa} + 1 \Bigr)^{1/2}L_J^{-s}
 .
\end{equation*}
\end{remark}
\begin{remark} 
 For smooth source terms $f\in C^\infty(\IS^2) = \bigcap_{s'>0} H^{-1+s'}(\IS^2)$, 
 the convergence rate of the MLMC estimator for Spectral Methods 
 is given by $s$ without further restrictions, cp.\ Theorem~\ref{thm:ellSPDEN} 
 and Remark~\ref{rmk:MLMC_spect_meth}.
 The decay of the angular power
 spectrum of the underlying iGRF~$T$ in terms of $\beta >0$ in~\eqref{eq:AngSpecDec}
 is the only constraint on the convergence rate $s$ since $s<\beta/2$.
 For Finite Element Methods, the convergence rate is additionally bounded by the polynomial degree $k$ of the Finite Element space,
 cp.\ Theorem~\ref{thm:MLMC_error_est_u^L_h}.
 
 We conclude that we have essentially determined the achievable convergence rates of MLMC FE and 
 Spectral Methods solely with the decay of the angular power spectrum
 of the underlying iGRF in the stochastic operator, which in the FE case are bounded by the polynomial degree of the basis functions.
\end{remark}
%
\begin{remark}
 There exists an algorithm to compute samples of an iGRF that has a complexity behaving as $\Op(N \log^2(N))$, cp.\ ~\cite{HRKM03}, 
 where $N$ is the number of sample points 
 of a quadrature to compute stiffness matrices. The number of sample points is comparable to the degrees of freedom of the spatial 
 discretization.
 In the FE case, iterative solvers such as multigrid, cp.\ ~\cite{Bramble_Multigrid}, suggest to have a complexity that is linear in the degrees of freedom,
 where here the resulting linear systems do not render the classical theory, since condition numbers of system matrices may be close to 
 degenerate due to the 
 lognormal diffusion coefficient.
 In the setting of Corollary~\ref{cor:err_and_work_bounds}, this would allow for $\eta_1=1$ and $\eta_2=2$.  
\end{remark}
\begin{remark}
Since the degrees of freedom of either of the considered spatial discretizations relate to the discretization parameter with
$N_{h_J} ={\rm dim}(V^{h,k}) = \Op(h_J^{-2})$ and $N_{L^u}= {\rm dim}(\mathcal{H}_{1:L^u})= \Op((L^u)^2)$, cp.\ Section~\ref{sec:Discr},
respective convergence estimates and work bounds from Corollary~\ref{cor:err_and_work_bounds} and Remark~\ref{rmk:MLMC_spect_meth} 
in the degrees of freedom are implied.
\end{remark}
\begin{remark}
A decrease of the choices of samples $(M_j,j=0,\dots,J)$ in~\eqref{eq:MLMC_levels},
i.e., if $\kappa<1$,
will increase the MC error contribution 
in Theorem~\ref{thm:MLMC_error_est_u^L_h} 
basically by the inverted square root of~$\kappa$ due to a larger MC error contribution.
For instance, applying \cite[Theorem~1]{L16} in our setting yields 
sample numbers scaled by a factor of $2^{-2\min\{s,k\}}$ for $j=1,\dots,J$ sacrificing an increase 
in the corresponding constant $C_2$ (in the notation of \cite[Theorem~1]{L16}) of the error estimate.
This constant will be scaled by a factor of $2^{2\min\{s,k\}}$.
\end{remark}
\section{Numerical experiments}\label{sec:num_exp}

We consider here the test problem with smooth right hand side $f=Y_{1 0}$, i.e., $\ell=1$ and $m=0$, and angular power spectrum given 
by
\begin{equation*}
 A_\ell = (1+\ell)^{-\alpha}, \qquad \ell \in \N_0,
\end{equation*}
for $\alpha\in (2,+\infty)$. 
Since 
\begin{equation*}
\sum_{\ell\geq 0}
A_\ell \ell^{1+\beta}<+\infty
\end{equation*}
for every $\beta< \alpha-2$, Theorem~\ref{thm:logNiGRFsProp} implies that the respective lognormal random field $a\in L^p(\Omega;C^{\iota,\gamma}(\IS^2))$,
$p\in[1,+\infty)$,
for every $\iota\in\N_0$ and $\gamma\in(0,1)$ such that $\iota+\gamma<(\alpha-2)/2$.
For a given number of levels $J\in\N_0$, we study the error $\E(u) - E^J(u^{L_J,h_J})$ in the $L^2(\Omega;H^1(\IS^2)/\R)$-norm. 
The sample numbers per level are chosen according to \eqref{eq:MLMC_levels}.
Also the truncation levels are chosen as mentioned in Section~\ref{sec:MLMC},
i.e., $L_j=\lceil h_0 / h_j\rceil L_0$, $j=0,\dots,J$, for some fixed $L_0\in\N$.
We therefore expect by Corollary~\ref{cor:err_and_work_bounds} to observe a convergence rate of $\Op(h_J^{\min\{s,k\}})$ 
for any $s<(\alpha-2)/2$, where $k$ is the polynomial degree of the ansatz functions.

The implementation of the MLMC estimator in \eqref{eq:MLMC_est_def} for the FE method presented in Section~\ref{sec:FEMS2}
builds on the structures of the boundary element C++ library \texttt{BETL}, cp.\ \cite{BETL}.
The geometrical error that occurs when surfaces are polynomially approximated can in the case of $\IS^2$ be avoided with the 
correction formula in
\cite[Equation~(2.12)]{Demlow2009} for the gradient and with the correction formula
in \cite[Equation~(2.10)]{Demlow2009} for the surface measure. 
The FE spaces result from refining an inscribed initial affine approximation of $\IS^2$.
New vertices are not projected to $\IS^2$ to obtain nested FE spaces. 
FE spaces on $\IS^2$ result by lifting functions to $\IS^2$, cp.\ \cite[Sections~2.4 and~2.5]{Demlow2009}.
This is feasible in the case of $\IS^2$, because a \emph{signed distance function} is explicitly known. 
A signed distance function of $\IS^2$ maps points in $\R^3\backslash\{0\}$ to their distances to $\IS^2$ multiplied
by a negative sign if (by convention) they are inside of the closed surface $\IS^2$.

The evaluation of the truncated spherical harmonics series \eqref{eq:PL} is implemented with the \texttt{SHTns} library,
cp.\ \cite{SHTns2013}. 
This implementation has a larger complexity of $\Op(N^{3/2})$ assuming that $N \sim L^2$ is the 
number of grid points compared to $\Op(N \log(N)^2))$, which is the complexity of the algorithm presented in \cite{HRKM03}.
However, the available implementation of the latter algorithm seems to require significant amounts of memory, cp.\ \cite{SHTns2013}, which limits the truncation level~$L$. 
Also, \texttt{SHTns} outperforms amongst others the available implementation of the algorithm in \cite{HRKM03} 
in measured computing time as demonstrated in \cite{SHTns2013} 
and in particular allows for higher truncation levels.
The linear systems are solved with the Intel \texttt{MKL} version of the \texttt{PARDISO} solver (see also \cite{Schenk200169}).
The evaluation of the MLMC estimator is parallelized with the \texttt{gMLQMC} library, cp.\ \cite{gMLQMC}, which allows for generic sampling.
The parallelization in \texttt{gMLQMC} uses the \texttt{Boost.MPI} library.
As a pseudo random number generator, we use the Mersenne Twister implementation from the \emph{C++11 standard}.
We will use \texttt{matplotlib}, cp.\ \cite{matplotlib}, to visualize our data (see Figure~\ref{fig:conv_plot}). 

We present numerical results for first order FE, i.e., $k=1$, for $\alpha=3$ and $\alpha=4$, which have 
a theoretical convergence rate of essentially $0.5$ and $1$, respectively.
The sample numbers on each level are chosen as in \eqref{eq:MLMC_levels} with $\kappa=1$, $\varepsilon = 0.2$, and $L_0=2$. 
As reference for $\E(u)$ we use the average of $10$ realizations of the MLMC FE estimator with one further level of refinement and 
parameter choices $\kappa=40$ and $L_0=5$.
The $L^2(\Omega;H^1(\IS^2)/\R)$-norm is approximated by the square root of the average over 20 realizations of $\|\E(u) - E^J(u^J)\|^2_{H^1(\IS^2/\R)}$.
In Figure~\ref{fig:conv_plot} we observe convergence rates that our theoretical 
analysis predicts, since for the border line cases of convergence rates equal to $0.5$ and $1$, constants in the error bounds may become 
arbitrarily large. The empirical rates have been computed with least squares 
    taking into account the five data points corresponding to finer spatial grids.
\begin{figure}[ht]
    \centering
    \includegraphics[width=0.7 \textwidth]{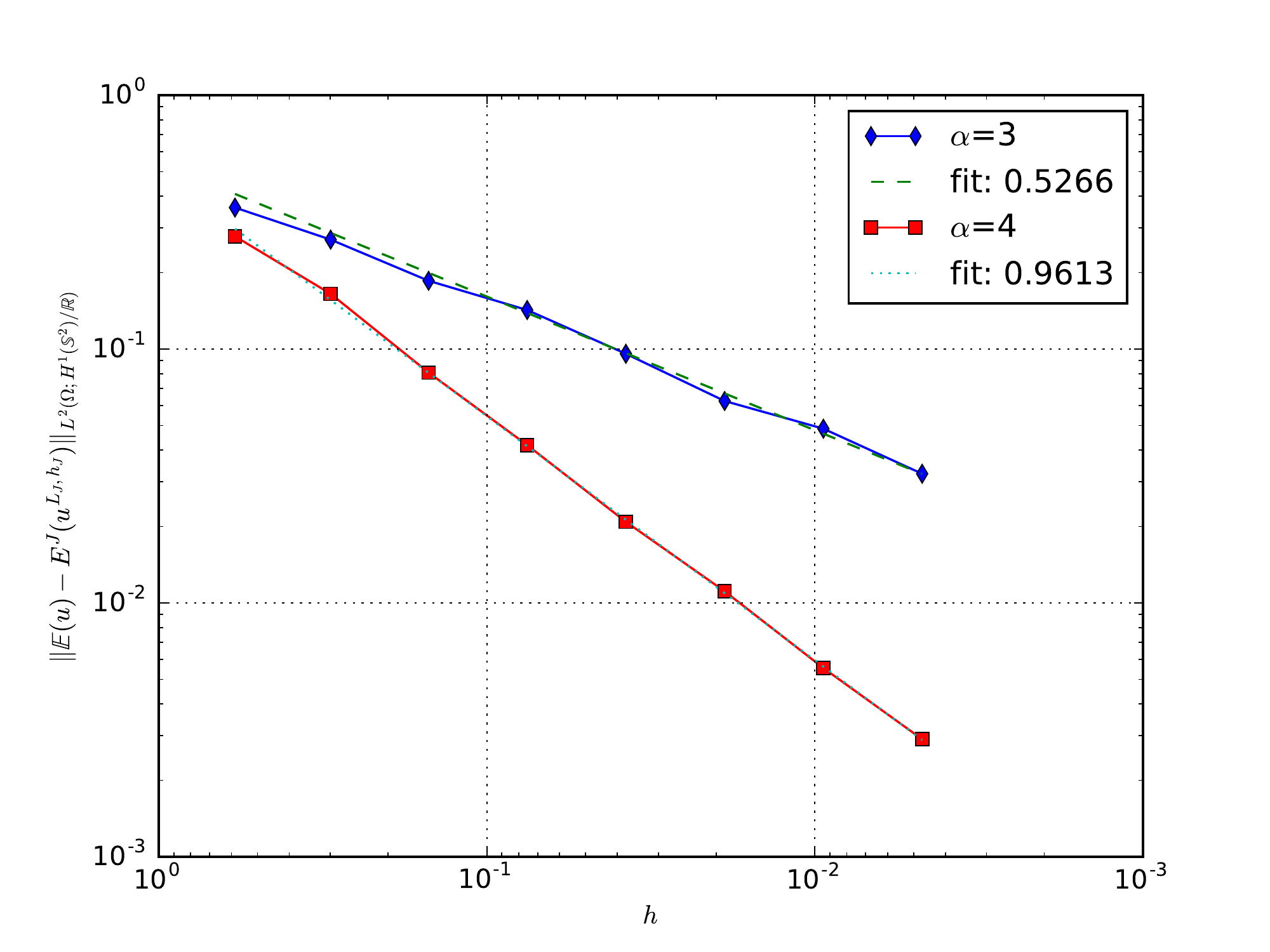}
    \caption{Convergence of MLMC with $\alpha= 3,4$ and $\kappa=1$.}
    \label{fig:conv_plot}
\end{figure}

\appendix

\section{Measure theory and functional analysis}\label{app:meas_func-ana}

The purpose of this appendix is to collect the measure theoretical and functional analytical background of our analysis in the main text. While results on domains in Euclidean space are well-known in the literature, the corresponding results on the unit sphere are not explicitly available. Therefore, we derive these missing results for the unit sphere in what follows, which include properties of H\"older and Sobolev spaces and especially a Sobolev embedding theorem on~$\IS^2$.

One way to translate results from Euclidean space to manifolds such as the unit sphere is to use an atlas and show the invariance of the results under a change of atlas. This will be of frequent use in what follows.
Therefore, let $\{(U_i,\eta_i),i\in\cI\}$ be a finite $C^\infty$ atlas of $\IS^2$, where 
$\{U_i,i\in\cI\}$ is a finite open cover of $\IS^2$ and $\{\eta_i:U_i\rightarrow \eta_i(U_i)\subset\R^2,i\in\cI\}$ 
are the respective coordinate charts, which are sometimes also simply called \emph{coordinates}.
Furthermore, let $g$ be the metric tensor which is expressed for any $x_0\in\IS^2$ locally 
in the coordinates $\{\eta_i,i\in\mathcal{I}\}$ as
\begin{equation*}
 g_{k\ell}(x_0)
 :=
 \Bigl\langle \frac{\partial \eta_i^{-1}(\hat{x}_0)}{\partial \hat{x}^k}  , 
 \frac{\partial \eta_i^{-1}(\hat{x}_0)}{\partial \hat{x}^\ell}  \Bigr\rangle_{\R^3}
\end{equation*}
for $k,\ell=1,2$, where $\hat{x}_0 = \eta_i(x_0)$ and $i\in\mathcal{I}$ is such that $x_0\in U_i$.
The matrix $g(x_0)$ induces an inner product on the 
\emph{tangent space $T_{x_0}\IS^2$ at $x_0$} 
in the basis
$\frac{\partial \eta_i^{-1}(\hat{x}_0)}{\partial \hat{x}^k} $, $k=1,2$, i.e., for 
$v=\sum_{k=1}^2 v^k \frac{\partial \eta_i^{-1}(\hat{x}_0)}{\partial \hat{x}^k} $,
$w=\sum_{k=1}^2 w^k \frac{\partial \eta_i^{-1}(\hat{x}_0)}{\partial \hat{x}^k} \in T_{x_0}\IS^2$, 
it holds that
$\langle v,w \rangle_{\R^3} = \sum_{k,\ell=1}^2 g_{k\ell}(x_0)v^k w^\ell$.
We denote the components of the inverse of $g$ at any arbitrarily chosen $x_0\in\IS^2$ by 
$g^{k\ell}(x_0) := (g^{-1}(x_0))_{k\ell}$ for $k,\ell=1,2$
and further introduce $|g|(x_0):=\det(g(x_0))$.
The \emph{spherical gradient} $\nabla_{\IS^2}$ and the \emph{spherical divergence} $\nabla_{\IS^2}\cdot$ 
are locally expressed in terms of~$g$, i.e., for any $x_0\in\IS^2$, $i\in\mathcal{I}$ 
such that for any $x_0\in U_i$ and $\hat{x}_0=\eta_i(x_0)$,
\begin{equation*}
 \nabla_{\IS^2} f(x_0)
 :=
 \sum_{k,\ell=1}^2 g^{k\ell}(x_0) \frac{\partial (f\circ\eta_i^{-1})(\hat{x}_0)}{\partial \hat{x}^k}
 \frac{\partial \eta_i^{-1}(\hat{x}_0) }{\partial \hat{x}^\ell}
\end{equation*}
and 
\begin{equation*}
 \nabla_{\IS^2} \cdot Z (x_0)
 :=
 \frac{1}{\sqrt{|g|(x_0)}} 
 \sum_{\ell=1}^2
 \frac{\partial}{\partial \hat{x}^\ell}
 ((\sqrt{|g|} Z^\ell)\circ\eta_i^{-1})(\hat{x}_0),
\end{equation*}
where $f:\IS^2\rightarrow\R$ is a function and 
$Z=\sum_{\ell=1}^2 Z^\ell \frac{\partial \eta_i^{-1}}{ \partial \hat{x}^\ell}$ a vector field,
cp.\ ~\cite[Equations (3.1.17), (3.1.19)]{jost2011RieGeoGeoAna6ed}.
We define the
\emph{spherical Laplacian}, also called 
\emph{Laplace--Beltrami operator}, by
\begin{equation*} 
 \Delta_{\IS^2} := \nabla_{\IS^2} \cdot \nabla_{\IS^2},
\end{equation*}
and we denote by $\gs$ the \emph{Lebesgue measure on the sphere}
which admits for every $i\in \cI$ the local representation 
  \begin{equation*}
   \dd\gs(x)
    = \sqrt{|g|(x)}\mathrm{d}\hat{x}^1\mathrm{d}\hat{x}^2
  \end{equation*}
on $U_i$ by~\cite[Equation (3.3.8)]{jost2011RieGeoGeoAna6ed},
where $x\in U_i$ and $\hat{x}=\eta_i(x)$.
These definitions are valid fo general coordinates and therefore generalize the 
respective expressions given in Section~\ref{sec:iGRF}.
Note that for any $x_0\in\IS^2$, 
the inner product that is induced in $T_{x_0}\IS^2$ by $g(x_0)$ 
does not depend on the choice of the coordinates $\{\eta_i,i\in\mathcal{I}\}$, 
cp.\ ~\cite[Equations (1.4.4), (1.4.5)]{jost2011RieGeoGeoAna6ed}.
For further details, the reader is referred 
to~\cite[Sections~1.4 and~3.1]{jost2011RieGeoGeoAna6ed}.

Furthermore, let $\Psi=\{\Psi_i,i\in\cI\}$ be a $C^\infty$ partition of unity, 
which is subordinate to $\{U_i,i\in\cI\}$, i.e., 
$\supp(\Psi_i)\subset U_i$ for every $i\in\cI$.
The support of a function is the closure of the points, where the function is non-zero.
We infer from~\cite[Theorem~7.4.5]{Triebel92} and~\cite[Theorem~3.9]{GS13} 
that Sobolev spaces on~$\IS^2$ can equivalently to Section~\ref{sec:iGRF} 
be characterized via pullbacks with respect to general coordinates, i.e.,
$v\in H^s_q(\IS^2)$ if and only if $(v\Psi_i)\circ\eta_i^{-1} \in H^s_q(\R^2)$ 
for every $i\in\cI$, 
where $v \Psi_i$ has to be understood as pointwise multiplication, and 
\begin{equation*}
 v\mapsto\Bigl( \sum_{i\in\cI} \|(v\Psi_i)\circ\eta_i^{-1}\|^q_{H^s_q(\R^2)}\Bigr)^{1/q}
\end{equation*}
is an equivalent norm on $H^s_q(\IS^2)$,
where $H^s_q(\R^2)$ denote the usual Bessel potential spaces on $\R^2$, which are equal to
the Sobolev--Slobodeckij spaces for $q=2$ with equivalent norms, cp.\ ~\cite[Definition~2.3.1(d), Theorem~2.3.2(d), Equation~4.4.1(8)]{Triebel95}.
More precisely, ~\cite[Theorem~7.4.5]{Triebel92} implies that $H^s_q(\IS^2)$ can be 
equivalently characterized via pullbacks with respect to the geodesic normal coordinates.
In~\cite[Theorem~3.9]{GS13} it is shown that the characterization of Sobolev spaces on manifolds with bounded geometry, 
e.g., $\IS^2$, via pullbacks with respect to arbitrary coordinates does not depend on the coordinates and that different coordinates lead to equivalent norms.
We remark that a function like $(v\Psi_i)\circ\eta_i^{-1}$ on $\eta_i(U_i)$ 
can be extended smoothly by zero to all of $\R^2$, 
since $\Psi_i\circ\eta_i^{-1}$ is smooth and compactly supported in $\eta_i(U_i)$.
For details on the geodesic normal coordinates, which are sometimes also called (Riemannian) normal coordinates 
(cp.\ ~\cite[Definition~1.4.4]{jost2011RieGeoGeoAna6ed}), 
we refer the reader to~\cite[Example~3]{GS13},
while a detailed description of Bessel potential spaces can be found in~\cite[Chapter~2]{T83}.
%


Finally, we equip the H\"older spaces~$C^{\iota,\gamma}(\IS^2)$, $\iota\in \N_0$, $\gamma\in[0,1)$, that were introduced in Section~\ref{sec:iGRF} with the norm $\|\cdot\|_{C^{\iota,\gamma}(\IS^2)}$ given by
\begin{equation*}
 \|v\|_{C^{\iota,\gamma}(\IS^2)}
 :=\max_{i\in\cI}\|(v\Psi_i)\circ\eta_i^{-1}\|_{C^{\iota,\gamma}(\R^2)}
\end{equation*}
for every $v \in C^{\iota,\gamma}(\IS^2)$.
This norm is well-defined, since different choices of atlases and partitions of unity 
will lead to equivalent norms (cp.\ ~\cite[Proposition~6.9]{Herrmann13}).

A nice and convenient property of the regularity of the exponential function in terms 
of H\"older norm bounds that will be introduced as the following lemma.
This lemma is proven by an induction argument using the fact that H\"older spaces are algebras, i.e.,
the product of functions is an element of the same H\"older space and the product is continuous. 

\begin{lemma}\label{lemma:Hoelder_prod_comp_est}
 Let $\iota\in\N_0$ and $\gg\in(0,1)$, 
 then there exists a constant $C_{\iota,\gg}$ such that for every $v\in C^{\iota,\gg}(\IS^2)$
\begin{equation}\label{est:Hoelder_prod_comp_est}
 \|\exp( v)\|_{C^{\iota,\gg}(\IS^2)}\leq C_{\iota,\gg} \|\exp(v)\|_{C^{0}(\IS^2)}\bigl(1+\|v\|_{C^{\iota,\gg}(\IS^2)}^{\iota+1}\bigr).
\end{equation}
\end{lemma}

\begin{proof}
Generally, this proof is inspired by~\cite[Theorem~A.8]{Hoermander1976}, 
but it achieves a specific result not explicitly available in that theorem.

Let $D\subset\R^2$ be a bounded, convex open domain. 
The first step is to prove the estimate for H\"older spaces over Euclidean domains, i.e., $C^{\iota,\gg}(\overline{D})$.
We set $g:=\exp$
and recall that the derivative $g'$ is again equal to $g$.

For convenience, we will omit the set $\overline{D}$ if the context is clear.
 For $\iota=0$, it is easily seen that for every $\tilde{v}\in C^{0.\gg}(\overline{D})$, 
it holds that 
 $\|g(\tilde{v})\|_{C^{0,\gg}}\leq \|g(\tilde{v})\|_{C^{0}}(1+\|\tilde{v}\|_{C^{0,\gg}})$, 
 cp.\ the proof of~\cite[Theorem~A.8]{Hoermander1976},
 which is the base case of an induction argument
 to the following induction hypothesis:
 
 Let the estimate in~\eqref{est:Hoelder_prod_comp_est} be satisfied 
 for H\"older spaces over the Euclidean set $\overline{D}$,
 i.e., for functions $\tilde{v}\in C^{n,\gg}(\overline{D})$
 for every $n\in\{0,\dots,\iota-1\}$.
 We directly perform the induction step from $n=\iota-1$ to $n+1 = \iota$.
 For $\tilde{v}_1,\tilde{v}_2\in C^{\iota.\gg}(\overline{D})$, the product estimate $\|\tilde{v}_1\tilde{v}_2\|_{C^{\iota,\gg}}\leq \hat{C}_{\iota,\gg}\|\tilde{v}_1\|_{C^{\iota,\gg}}\|\tilde{v}_2\|_{C^{\iota,\gg}}$ holds by~\cite[Theorem~A.7]{Hoermander1976},
 which implies with the chain rule from calculus and the induction hypothesis that
 \begin{align}\label{est:Hoelder_prod_comp_est_R2}
 \begin{split}
  \|g( \tilde{v})\|_{C^{\iota,\gg}}
  &=\|g( \tilde{v})\|_{C^{0}} + \sum_{j=1,2} \|\partial_{x^j}(g\circ \tilde{v})\|_{C^{\iota-1,\gg}}\\
  &\leq \|g( \tilde{v})\|_{C^{0}} + \hat{C}_{\iota-1,\gg}\sum_{j=1,2} \|g'( \tilde{v})\|_{C^{\iota-1,\gg}}\|\partial_{x^j}\tilde{v}\|_{C^{\iota-1,\gg}}\\
  &\leq \|g( \tilde{v})\|_{C^{0}} + \hat{C}_{\iota-1,\gg} \|g( \tilde{v})\|_{C^{\iota-1,\gg}}\|\tilde{v}\|_{C^{\iota,\gg}}\\
  &\leq \|g( \tilde{v})\|_{C^{0}} + \hat{C}_{\iota-1,\gg}\widetilde{C}_{\iota-1,\gg}\|g( \tilde{v})\|_{C^{0}}(1+\|\tilde{v}\|_{C^{\iota-1,\gg}}^\iota)\|\tilde{v}\|_{C^{\iota,\gg}}\\
  &\leq \widetilde{C}_{\iota,\gg}\|g( \tilde{v})\|_{C^{0}}(1+\|\tilde{v}\|_{C^{\iota,\gg}}^{\iota+1}),
 \end{split}
 \end{align}
 and finishes the induction step.
 Note that for convenience we used $\partial_{x^j} := \partial/\partial x^j$, $j=1,2$.
 
Next let $\{(U_i,\eta_i),i\in\cI\}$ be a finite $C^\infty$ atlas and $\{\Psi_i,i\in\cI\}$ a $C^\infty$ partition of unity subordinate to $\{U_i,i\in\cI\}$.
We fix $j\in\cI$ and choose another $C^\infty$ partition of unity $\{\hat{\Psi}_i,i\in\cI\}$ subordinate to $\{U_i,i\in\cI\}$ such that $\hat{\Psi}_j=1$ on $\supp(\Psi_j)$.
We can assume that $\overline{D}:=\supp(\Psi_j\circ\eta_j^{-1})$
and $\supp(\hat{\Psi}_j\circ\eta_j^{-1})$ are convex and observe with~\eqref{est:Hoelder_prod_comp_est_R2} that
\begin{align}\label{est:Hoelder_prod_comp_est_S2}
\begin{split}
 \|(\exp( v)\Psi_j)\circ\eta_j^{-1}\|_{C^{\iota,\gg}(\R^2)}
 &=\|(\exp( v)\Psi_j)\circ\eta_j^{-1}\|_{C^{\iota,\gg}(\overline{D})}\\
 &\leq \hat{C}_{\iota,\gg} \|\exp( v)\circ\eta_j^{-1}\|_{C^{\iota,\gg}(\overline{D})}\|\Psi_j\circ\eta_j^{-1}\|_{C^{\iota,\gg}(\overline{D})}\\
 &\leq C_{\iota,\gg} \|\exp( v)\circ\eta_j^{-1}\|_{C^{0}(\overline{D})}(1+ \|v\circ\eta_j^{-1}\|_{C^{\iota,\gg}(\overline{D})}^{\iota+1})\\
 &\leq C_{\iota,\gg} \|(\exp( v)\hat{\Psi}_j)\circ\eta_j^{-1}\|_{C^{0}(\R^2)}\bigl(1+ \|(v\hat{\Psi}_j)\circ\eta_j^{-1}\|_{C^{\iota,\gg}(\R^2)}^{\iota+1}\bigr).
\end{split}
\end{align}
We apply that different $C^\infty$ partitions of unity result 
in equivalent norms on $C^{\iota,\gg}(\IS^2)$ and conclude the estimate of the 
lemma by taking the maximum over $j$ on both sides of~\eqref{est:Hoelder_prod_comp_est_S2}.
\end{proof}

As in Euclidean space, Sobolev spaces can be embedded into H\"older spaces (see e.g.\ \cite[Theorem~4.6.1(e)]{Triebel95}) which is made precise in the following Sobolev embedding theorem on~$\IS^2$.
\begin{theorem}[Sobolev embedding theorem]\label{thm:Sobelev_emb}
If $s\in(0,+\infty), q\in(1,+\infty), \iota\in\N_0$, and $\gg\in(0,1)$ satisfy
$s-2/q\geq\iota+\gg$, then the embedding $H^s_{q}(\IS^2)\subset C^{\iota,\gg}(\IS^2)$
is continuous.
\end{theorem}

Furthermore, $H^s_q(\IS^2)$-norms of products of functions can be bounded by a combination of H\"older and Sobolev norms, which is made in the following proposition. In the proof, the estimate for domains in Euclidean space in~\cite[Theorem~3.3.2]{T83} is translated to~$\IS^2$.
\begin{proposition}\label{prop:Hoelder_Sobolev_prod}
Let $q\in(1,+\infty)$ and let $\iota\in\N_0$, $\gamma\in(0,1)$, and $s\in \R$ be
such that $|s| < \iota+\gamma$.
If $v\in C^{\iota,\gg}(\IS^2)$ and $w\in H^s_q(\IS^2)$, then $vw\in H^s_q(\IS^2)$.
Moreover the following product estimate holds: there exists a constant $C_{\iota,\gg}$ such that 
for every $v\in C^{\iota,\gamma}(\IS^2)$ and 
every $w\in H^{s}_q(\IS^2)$,
\begin{equation*}
\|vw\|_{H^s_q(\IS^2)}\leq C_{\iota,\gg} \|v\|_{C^{\iota,\gamma}(\IS^2)}\|w\|_{H^s_q(\IS^2)}.
\end{equation*}
\end{proposition}
The proofs of Theorem~\ref{thm:Sobelev_emb} and Proposition~\ref{prop:Hoelder_Sobolev_prod} 
follow with a localization argument as applied in the second paragraph of the proof of Lemma~\ref{lemma:Hoelder_prod_comp_est}.

Let us conclude this appendix with some facts about measurability of Banach-space-valued random variables, which includes our framework of Sobolev and H\"older spaces. 
Therefore, consider a Banach space $(B,\|\cdot\|_B)$ with dual space $B^*$ and $X:\gO\rightarrow B$.
We recall that $X$ is called \emph{weakly measurable} if for every $\cG\in B^*$, 
the real-valued function $\cG(X)$ is measurable.
Furthermore, $X$ is called \emph{countably-valued} if $X$ assumes at most a countable set of values in $B$ on countably many, disjoint measurable subsets.
It is \emph{strongly measurable} if there exists a sequence of countably-valued mappings $(X_n,n\in\N)$, $X_n: \gO \rightarrow B$, such that $\lim_{n\to+\infty}X_n(\go)=X(\go)$ in $B$ for $\IP$-a.e.\ $\go\in\gO$,
and we say that $X$ is called \emph{$\IP$-almost separably-valued} if there exists a measurable set $N$ with $\IP(N)=0$ such that the set $\{X(\go),\go\in\gO\backslash N\}$ is separable in $B$ (cp.\ ~\cite[Definitions~3.5.3 and~3.5.4]{hille1974Funanasem}).
A well-known result on the connection of strong and weak measurability is Pettis' theorem (see, e.g., \cite[Theorem~3.5.3]{hille1974Funanasem}).
\begin{theorem}[Pettis' theorem]\label{thm:Pettis}
A $B$-valued mapping on $\gO$ is strongly measurable 
if and only if it is weakly measurable and $\IP$-almost separably-valued.
\end{theorem}

The following lemma is the generalization to Banach spaces of the well-known property that real-valued random variables under continuous mappings are random variables, i.e., measurable. It is a direct consequence of the definition of strong measurability. 
 \begin{lemma}\label{lemma:comp_meas_maps}
  Let $B_1,B_2$ be Banach spaces and let $\varphi:B_1\rightarrow B_2$ 
  be continuous. 
  If $f:\gO\rightarrow B_1$ is strongly $B_1$-measurable, then $\varphi\circ f$
  is strongly $B_2$-measurable.
 \end{lemma}
Since we consider in this manuscript measurability with respect to different Banach spaces, we write for clarity $B$-measurable where necessary.
We remark that a mapping $X:\gO\rightarrow B$
is Bochner integrable if and only if it is strongly $B$-measurable and 
the real-valued function $\|X\|_B$ is integrable, 
cp.\ ~\cite[Theorem~3.7.4]{hille1974Funanasem}.
The strong $B$-measurability of $X$ 
implies the measurability of~$\|X\|_B$.

\section{Integrability of continuous lognormal RFs}

Integrability of a lognormal random field in terms of $L^p(\gO)$-norms is a consequence of Fernique's theorem. While this was performed for random fields on domains in Euclidean space in~\cite[Proposition~3.10]{Charrier_SINUM12}, we derive the corresponding result on spheres in this appendix in the following proposition.

\begin{proposition}\label{prop:iLogRF_C^0bound}
Let $p\in[1,+\infty)$ and let $T$ be a continuous iGRF and satisfy \eqref{eq:AngSpecDec} for some $\gb>0$.
Then, the lognormal random fields $\exp(T)$ and $\exp(\Pi_L T)$ are in $L^p(\gO;C^0(\IS^2))$ 
for all $L\in\N_0$ and the $L^p(\gO;C^0(\IS^2))$-norm 
of $\exp(\Pi_L T)$ can be bounded independently of~$L$.
\end{proposition}

\begin{proof}
  It will be sufficient to prove the case that $T$ is centered, i.e., $\E(T) = 0 \in C^0(\IS^2)$.
  By the definition of an iGRF $\E(T)$ is a constant function on $\IS^2$,
  which implies 
  $\|\exp(T)\|_{L^p(\gO;C^0(\IS^2))} = \|\exp(T-\E(T))\|_{L^p(\gO;C^0(\IS^2))}\exp(\E(T))$.
  Hence, the general case can be reduced to the case of a centered, continuous iGRF.
  So in the following we can assume that $T$ is centered.
  
  The idea of the proof is to apply Fernique's theorem, cp.\ ~\cite[Theorem~2.7]{DPZ2014}, 
  on the separable Banach space $C^0(\IS^2)$. 
  Therefore, we have to establish that the law of $T$ is a centered 
  (symmetric) Gaussian measure on $C^0(\IS^2)$, 
  i.e., for every $\cG\in C^0(\IS^2)^*$, the dual space of $C^0(\IS^2)$, 
  there exists $\sigma_\cG\in[0,+\infty)$ such that $\cG(T) \sim \cN(0,\sigma^2_\cG)$.
  This is the first requirement in order to apply~\cite[Theorem~2.7]{DPZ2014}.
  We remark that in~\cite{DPZ2014} the term 'symmetric' Gaussian measure is used instead of centered meaning the same. 
  For every $L\in\N_0$, $\Pi_L T$ has the finite real expansion according to~\eqref{eq:real_exp}
  \begin{equation*}
  \Pi_L T=
   \sum_{\ell=0}^L \Bigl( a_{\ell 0} Y_{\ell 0}
    + 2 \sum_{m=1}^\ell \left(\Re a_{\ell m} \Re Y_{\ell m} - \Im a_{\ell m} \Im Y_{\ell m}\right)\Bigr).
  \end{equation*}
From the properties of the \KL expansion we deduce that $\{a_{\ell 0},\Re a_{\ell m},\Im a_{\ell m},\ell\in\N_0,m=1,\dots,\ell\}$ 
are independent real-valued random variables. 
Additionally this corollary implies that $a_{\ell 0}\sim \cN(0,A_\ell)$ and $\Re a_{\ell m}$, $\Im a_{\ell m}\sim \cN(0,A_\ell/2)$ for $\ell\in\N_0$ and $m=1,\dots,\ell$.
Let $L\in\N_0$ and $\cG\in C^0(\IS^2)^*$ be arbitrary.
Hence,
\begin{equation*}
 \cG(\Pi_L T)= \sum_{\ell=0}^L \Bigl( a_{\ell 0} \cG(Y_{\ell 0})
    + 2 \sum_{m=1}^\ell \left(\Re a_{\ell m} \cG(\Re Y_{\ell m}) - \Im a_{\ell m} \cG(\Im Y_{\ell m})\right)\Bigr)\sim \cN(0,\sigma^2_{\cG,L})
\end{equation*}
and therefore the characteristic function $\varphi_{\cG,L}$ of $\cG(\Pi_L T)$ is given by
\begin{equation*}
\lambda \mapsto \varphi_{\cG,L}(\lambda) 
      := \exp\Bigl(-\frac{1}{2} \lambda^2 \sigma^2_{\cG,L}\Bigr), 
\end{equation*}
where 
\begin{equation*}
 \sigma^2_{\cG,L}
 =\sum_{\ell=0}^L A_\ell\Bigl(\cG(Y_{\ell 0})^2 + 2\sum_{m=1}^\ell ( \cG(\Re Y_{\ell m})^2 + \cG(\Im Y_{\ell m})^2 )\Bigr).
\end{equation*}
Thus,
$\Pi_L T$ is a centered Gaussian measure on $C^0(\IS^2)$ for every $L\in\N_0$.
The next step is to show that the sequence 
$(\sigma^2_{\cG,L},L\in\N_0)$ is uniformly bounded.
The Riesz representation theorem 
for $C^0(\IS^2)$ (cp.\ ~\cite[Theorem~7.10.4]{Bogachev07}) 
and~\cite[Theorem~3.1.1, Remark~3.1.5]{Bogachev07} imply 
that there exist a finite, positive measure $\nu$
on $(\IS^2,\cB(\IS^2))$ and a measurable function $g$ satisfying 
$|g(x)| =1$ for every $x \in \IS^2$ such that
$\cG(v)=\int_{\IS^2} v g \dd \nu$ for every $v\in C^0(\IS^2)$, 
which implies with the Cauchy--Schwarz inequality that for every $v\in C^0(\IS^2)$,
\begin{equation*}
 \cG(v)^2 
 =\Bigl(\int_{\IS^2}vg \dd\nu\Bigr)^2 
 \leq \|v\|^2_{L^2(\IS^2,\nu)}\|g\|^2_{L^2(\IS^2,\nu)}
 = \|v\|^2_{L^2(\IS^2,\nu)}\nu(\IS^2).
\end{equation*}
This implies with the identity 
$\sum_{|m| \leq \ell} \left| Y_{\ell m}(x) \right|^2 = (2\ell+1)/(4\pi)$ (cp.\ ~\cite[Theorem~2.4.5]{Ned2001})
that
\begin{align*}
 &\cG(Y_{\ell 0})^2 + 2\sum_{m=1}^\ell ( \cG(\Re Y_{\ell m})^2 + \cG(\Im Y_{\ell m})^2 )\\
 &\qquad\leq\Bigl( \int_{\IS^2} Y_{\ell 0}^2 + 2\sum_{m=1}^\ell( (\Re Y_{\ell m})^2 + (\Im Y_{\ell m})^2 )\dd\nu\Bigr) \nu(\IS^2)\\
 &\qquad=\Bigl(\int_{\IS^2} \sum_{m=-\ell}^\ell |Y_{\ell m}|^2 \dd \nu\Bigr) \nu(\IS^2)
 = \nu(\IS^2)^2 \frac{2\ell+1}{4\pi}.
\end{align*}
Summing the previous inequality over $\ell$ 
implies with the finiteness of $ \sum_{\ell\geq0} A_\ell \frac{2\ell+1}{4\pi}$
that $(\sigma^2_{\cG,L},L\in\N_0)$ is uniformly bounded in $L$.
Hence, there exists a unique $\sigma_\cG\in[0,+\infty)$ such that $\sigma^2_{\cG,L}\rightarrow \sigma^2_\cG$ as $L\to +\infty$.
Thus, 
$\lim_{L\to +\infty}\varphi_{\cG,L}(\lambda)=\exp(-1/2\;\lambda^2\sigma^2_\cG)=:\varphi_\cG(\lambda)$ 
for every $\lambda\in\R$. 
The $L^2(\gO;C^0(\IS^2))$-convergence of $\Pi_L T \rightarrow T$, 
which is implied by Theorem~\ref{thm:LpRegIsoGRF},
yields that $\cG(\Pi_L T) \rightarrow \cG(T)$ in $L^2(\gO)$ and thus in distribution.
L\'evy's continuity theorem, cp.\ ~\cite[Theorem~IV.13.2.B]{loeve1977ProtheI}, 
implies that $\cG(T)\sim \cN(0,\sigma_\cG^2)$ 
and we conclude that the law of $T$ is a centered (symmetric) Gaussian measure on $C^0(\IS^2)$.

We infer from Theorem~\ref{thm:LpRegIsoGRF} that there exists an upper bound $K$ of the $L^2(\gO;C^0(\IS^2))$-norm of $T$ and of $\Pi_L T$, 
$L\in\N_0$, which is uniform in $L$.
Let in the following $X\in\{T,\Pi_L T , L\in\N_0\}$.
We choose $x_0\in[1/(1+\exp(-2)),1)$, which implies that $\log((1-x_0)/x_0)\leq-2$,
and set $r_0:=K/\sqrt{1-x_0}$.
We use the Chebychev inequality to obtain that
\begin{equation*}
 1-\IP(\|X\|_{C^0(\IS^2)}\leq r_0)
  =\IP(\|X\|_{C^0(\IS^2)}>r_0) 
  \leq \frac{\E(\|X\|_{C^0(\IS^2)}^2)}{r_0^2}
  \le \frac{K^2}{r_0^2}
  =1-x_0,
\end{equation*}
which implies that $\IP(\|X\|_{C^0(\IS^2)}\leq r_0)\geq x_0$.
We choose $\lambda>0$ such that $\lambda \le (1-x_0)/(32K^2)$, which implies that $32\lambda r_0^2\leq 1$,
and arrive with the monotonicity of the logarithm at the inequality
\begin{equation*}
 \log\Bigl(\frac{1-\IP(\|X\|_{C^0(\IS^2)}\leq r_0)}{\IP(\|X\|_{C^0(\IS^2)}\leq r_0)}\Bigr) + 32\lambda r_0^2\leq \log\Bigl(\frac{1-x_0}{x_0}\Bigr) + 32\lambda r_0^2 \leq -1
 .
\end{equation*}
This is the second requirement for~\cite[Theorem~2.7]{DPZ2014}. Since $X$ is a centered Gaussian measure on $C^0(\IS^2)$, ~\cite[Theorem~2.7]{DPZ2014} 
implies that
\begin{equation*}
 \E(\exp(\lambda\|X\|^2_{C^0(\IS^2)}))\leq \exp(16\lambda r_0^2) + \frac{\exp(2)}{\exp(2)-1},
\end{equation*}
which is a bound that is independent of $L$, because the choices of $r_0$ and $\lambda$ do not depend on $L$ due to the uniformity of the bound $K$.
Since $0\leq (\sqrt{\lambda}x - p/(2\sqrt{\lambda}))^2$ 
implies that $px\leq \lambda x^2 + p^2/(4\lambda)$ for every $x\in\R$, we conclude that
\begin{equation*}
\E(\|\exp(X)\|_{C^0(\IS^2)}^p)\leq \E(\exp(p\|X\|_{C^0(\IS^2)}))
\leq \E(\exp(\lambda\|X\|^2_{C^0(\IS^2)})\exp\Bigl(\frac{p^2}{4\lambda}\Bigr),
\end{equation*}
which finishes the proof of the proposition.
\end{proof}

\section{Higher order regularity of solutions}
\label{app:proof_reg_est}

In this appendix we present the proof of Proposition~\ref{prop:solution_map_regularity}, which we divide for better readability into one lemma and two propositions.
Let us start with the $H^{1+s}(\IS^2)$-regularity of the solution for $s \in [0,1)$. This is derived 
with a classical regularity estimate, which in the case of domains in Euclidean space is
due to Hackbusch, cp.\ ~\cite[Theorem~9.1.8]{hackbusch2010Elldifequ} 
(see also~\cite{GilTrud2ndEd}). Here we transfer the problem to Euclidean space and back with an atlas and a partition of unity.

\begin{lemma}\label{lemma:reg_est}
 For some $0\leq s<\gamma< 1$, let $\tilde{u}\in H^1(\IS^2)/\R$, $f\in H^{-1+s}(\IS^2)$, and 
 $\tilde{a}\in C^{0,\gamma}(\IS^2)\cap C^0_+(\IS^2)$ 
 satisfy the variational problem~\eqref{eq:ellPDEweak}
 %
then $\tilde{u}\in H^{1+s}(\IS^2)$ and there exists a constant $C$, 
which is independent of $\tilde{u},f$, and $\tilde{a}$, such that
\begin{equation*}
 \|\tilde{u}\|_{H^{1+s}(\IS^2)} 
 \leq 
 C 
\Bigl(
 \frac{1}{\min_{x\in\IS^2}\tilde{a}(x)}
 ( \|\tilde{a}\|_{C^{0,\gamma}(\IS^2)}  \|\tilde{u}\|_{H^{1}(\IS^2)}
 + \|f\|_{H^{-1+s}(\IS^2)})
 + \|\tilde{u}\|_{H^1(\IS^2)}\Bigr)
 .
\end{equation*}
\end{lemma}
\begin{proof}
 Let $\{(\eta_j,U_j),j\in\mathcal{I}\}$ be a $C^\infty$-atlas and $\{\Psi_j,j\in\mathcal{I}\}$
 be a subordinate, $C^\infty$ partition of unity. 
 Let us fix $i\in\mathcal{I}$.
 We observe with the product rule, i.e., $\nabla\cdot(v W) = \nabla v\cdot W + v\nabla\cdot W$ 
 for scalar and vector fields $v$ and $W$, 
 the divergence theorem, cp.\ \cite[Equation (2.4.185)]{Ned2001}, 
 and~\eqref{eq:ellPDEweak} that 
 $\tilde{u}\Psi_i$ satisfies for every $v\in H^1(\IS^2)/\R$ that
 \begin{align*}
  (\tilde{a}\nabla_{\IS^2}(\tilde{u}\Psi_i),\nabla_{\IS^2} v )
  &=
  (\tilde{a}\nabla_{\IS^2}\tilde{u}, \nabla_{\IS^2}(v\Psi_i))
  -(\tilde{a}\nabla_{\IS^2}\tilde{u}\cdot\nabla_{\IS^2}\Psi_i,v)
  +(\tilde{a}\tilde{u}\nabla_{\IS^2}\Psi_i, \nabla_{\IS^2} v)\\
  &=
  f(v\Psi_i - \frac{1}{|\IS^2|}\int v\Psi_i\mathrm d \sigma )
  -(\tilde{a}\nabla_{\IS^2}\tilde{u}\cdot\nabla_{\IS^2}\Psi_i,v)
  +(\tilde{a}\tilde{u}\nabla_{\IS^2}\Psi_i, \nabla_{\IS^2} v)\\
  &=
  f(v \Psi_i) - \frac{1}{|\IS^2|} f(1)(\Psi_i,v )
  -(\tilde{a}\nabla_{\IS^2}\tilde{u}\cdot\nabla_{\IS^2}\Psi_i,v)
  +(\tilde{a}\tilde{u}\nabla_{\IS^2}\Psi_i, \nabla_{\IS^2} v) 
  ,
 \end{align*}
where we remark that for every $v\in H^1(\IS^2)/\R$, it holds that
$v\Psi_i - 1/{|\IS^2|}\int v\Psi_i\mathrm d \sigma \in H^1(\IS^2)/\R$.
Let $V_i:=\eta_i(U_i)$ and let $D\subset\R^2$ with smooth boundary 
be such that $\supp(\Psi_i\circ\eta_i^{-1})\subset\subset D \subset\subset V_i$.
We recall that for two functions $w_1,w_2:\IS^2\rightarrow\R$, the first fundamental form 
of their gradients satisfies with respect to the coordinate chart $\eta_i$ that on $V_i$ it holds that
\begin{equation*}
 (\nabla_{\IS^2} w_1\cdot\nabla_{\IS^2} w_2)\circ\eta_i^{-1} 
 = \sum_{k,\ell=1}^2
 g^{k\ell} \circ\eta_i^{-1}\frac{\partial (w_1\circ\eta_i^{-1}) }{\partial x^k}\frac{\partial (w_2\circ\eta_i^{-1}) }{\partial x^\ell}
 .
\end{equation*}
Furthermore, there exists a constant $\lambda_g>0$ such that for every $y\in U_i$, 
$\sum_{k,\ell=1}^2g^{k\ell}(y) \xi_k\xi_\ell\geq\lambda_g \sum^2_{k=1}\xi_k^2$
for every $\xi\in T_{y}\IS^2$. 
We also recall that with respect to the coordinate chart $\eta_i$ it holds that 
$\mathrm d\sigma(y)=\sqrt{|g|(y)}\mathrm d x$, where $|g|(y)=\det{(g(y))}$ and $y=\eta_i^{-1}(x)$, and $|g|(y)>0$ for every $x\in V_i$.
We choose $\chi\in C^\infty(\R^2)$ such that $\chi=1$ on $\supp(\Psi_i\circ\eta_i^{-1})$ 
and $\chi=0$ on the complement of $D$.
We define the matrix-valued function
\begin{equation*}
 A := 
 \begin{cases}
  ((\sqrt{|g|}\tilde{a}g^{-1})\circ\eta_i^{-1})\;\chi + \min_{y\in U_i}\{\sqrt{|g|(y)}\tilde{a}(y)\}\lambda_g\;(1-\chi)\Id_{\R^2}
  &\quad \text{on $V_i$} \\
  \min_{y\in U_i}\{\sqrt{|g|}\tilde{a}\}\lambda_g\;(1-\chi)\Id_{\R^2}
  &\quad \text{else}
 \end{cases}
\end{equation*}
and the functions
\begin{align*}
 b := \begin{cases}  
     ((\sqrt{|g|}\tilde{a})\circ\eta_i^{-1})\chi\;
    \sum_{k,l=1}^2
    g^{k\ell}\circ\eta_i^{-1} \frac{\partial (\tilde{u}\circ\eta_i^{-1}) }{\partial x^k}\frac{\partial (\Psi_i\circ\eta_i^{-1}) }{\partial x^\ell}
    & \text{ on $V_i$}\\
    \chi
    & \text{ else} \end{cases}
 ,\\
 c := \begin{cases}
    ((\sqrt{|g|}\tilde{a}\tilde{u})\circ\eta_i^{-1})\;\chi\;
    \sum_{k=1}^2
    (g^{k1}, g^{k2})^\top \circ\eta_i^{-1} \frac{\partial (\Psi_i\circ\eta_i^{-1}) }{\partial x^k}  
    & \text{ on $V_i$}\\
    (\chi,\chi)^\top 
    & \text{ else } \end{cases}
    .
\end{align*}
We use these three functions to define the functional $F$ for every $w\in H^{1}(\R^2)$ by
\begin{equation}\label{eq:appendix_reg_est_RHS_def}
 w \mapsto F(w)
 :=
 f(((w\chi)\circ\eta_i)\Psi_i) 
 - \frac{1}{|\IS^2|}f(1)(\Psi_i,(w\chi)\circ\eta_i)
 - \int_{\R^2} b w\; \mathrm d x
 + \int_{\R^2} c \cdot \nabla w\; \mathrm d x
 .
\end{equation}
We observe that for every $w\in H^1(\R^2)$, the function $((w\chi)\circ\eta_i)$ can be extended to a function $\tilde{w}\in H^1(\IS^2)/\R$, 
which then satisfies that
\begin{equation*}
 F(w)
 =
 f(\tilde{w} \Psi_i) - \frac{1}{|\IS^2|} f(1)(\Psi_i,\tilde{w} )
  -(\tilde{a}\nabla_{\IS^2}\tilde{u}\cdot\nabla_{\IS^2}\Psi_i,\tilde{w})
  +(\tilde{a}\tilde{u}\nabla_{\IS^2}\Psi_i, \nabla_{\IS^2} \tilde{w})
\end{equation*}
and 
\begin{equation*}
  \int_{V_i} A\nabla ((u\Psi_i)\circ\eta_i^{-1}) \cdot \nabla w\; \mathrm d x
 = (\tilde{a}\nabla_{\IS^2}(u\Psi_i),\nabla_{\IS^2} \tilde{w})
 ,
\end{equation*}
where we used that $\chi=1$ on $\supp(\Psi_i\circ\eta_i^{-1})$.
Since $\supp(\Psi_i\circ\eta_i^{-1})\subset V_i$, we obtain that
\begin{equation}\label{eq:appendix_reg_est_PDE_R2}
 \int_{\R^2} A\nabla ((u\Psi_i)\circ\eta_i^{-1}) \cdot \nabla w\; \mathrm d x
 =
 F(w)
 \quad \forall w\in H^1(\R^2)
 .
\end{equation}
We now aim to prove finiteness of the $H^{-1+s}(\R^2)$-norm of $F$ and to find a suitable bound.
Let $\{\hat{\Psi}_j,j\in\mathcal{I}\}$ be another partition of unity subordinate to the open cover
$\{U_j,j\in\mathcal{I}\}$ such that $\hat{\Psi}_i\circ\eta_i^{-1}=1$ on $\supp(\chi)\supset \supp(\Psi_i\circ\eta_i^{-1})$,
which necessarily implies that $\hat{\Psi}_j=0$ on $\supp(\Psi_i)$ for every $j\neq i$.
Thus we obtain with the characterization of the $H^{1-s}(\IS^2)$-norm on chart domains, 
the partition of unity property of  $\{\hat{\Psi}_j,j\in\mathcal{I}\}$,
Proposition~\ref{prop:Hoelder_Sobolev_prod}, and~\cite[Theorem~3.3.2(ii)]{T83} 
that there are constants $C_1,C_2,C_3$ such that 
for every $w\in H^{1-s}(\R^2)$, it holds that
\begin{align*}
 |f(((w\chi)\circ\eta_i)\Psi_i)|
 &\leq C_1
 \|f\|_{H^{-1+s}(\IS^2)}
 \|(((w\chi)\circ\eta_i)\hat{\Psi}_i)\circ\eta_i^{-1}\|_{H^{1-s}(V_i)}
 \|\Psi_i\|_{C^1(\IS^2)}
 \\
 &\leq C_2
 \|f\|_{H^{-1+s}(\IS^2)}
 \|w\chi\|_{H^{1-s}(V_i)}
 \|\hat{\Psi}_i\circ\eta_i^{-1}\|_{C^1(\overline{V_i})}
 \\
 &\leq C_3
 \|f\|_{H^{-1+s}(\IS^2)}
 \|w\|_{H^{1-s}(\R^2)}
 \|\chi\|_{C^1(\overline{V_i})}
 .
\end{align*}
The fourth summand in the definition of $F$ in~\eqref{eq:appendix_reg_est_RHS_def} can be written in a distributional sense as 
$w\mapsto -\int_{\R^2}(\nabla\cdot c )w \;\mathrm dx$, where we applied that $c$ is compactly supported in $V_i$.
Note that for $\ell=1,2$ and $s\in\R$, the linear operators $\frac{\partial}{\partial x^\ell}:H^{s}(\R^2)\rightarrow H^{s-1}(\R^2)$
are bounded.
Hence, we conclude as in the proof of Proposition~\ref{prop:Hoelder_Sobolev_prod} 
with~\cite[Theorem~3.3.2(ii)]{T83}
and the property that $\chi=1$ on $\supp(\hat{\Psi}_i\circ\eta_i^{-1})$
that there exist constants $C_1,C_2,C_3,C_4$ such that
\begin{align*}
 \|\nabla\cdot c\|_{H^{-1+s}(\R^2)}
 \leq
 C_1 \|c\|_{H^s(\R^2)}
 &\leq 
 C_2 
 \|(\tilde{a}\circ\eta_i^{-1})\chi\|_{C^{0,\gamma}(\overline{V_i})}
 \|\Psi_i\circ\eta_i^{-1}\|_{C^1(\overline{V_i})}
 \|(\tilde{u}\circ\eta_i^{-1})\chi\|_{H^s(V_i)}
 \\
 &\leq
 C_3
 \|(\tilde{a}\hat{\Psi}_i)\circ\eta_i^{-1}\|_{C^{0,\gamma}(\overline{V_i})}
 \|(\tilde{u}\hat{\Psi}_i)\circ\eta_i^{-1}\|_{H^s(V_i)}
 \|\chi\|^2_{C^{0,\gamma}(\overline{V_i})}
 \\
 &\leq
 C_4
 \|\tilde{a}\|_{C^{0,\gamma}(\IS^2)}
 \|\tilde{u}\|_{H^s(\IS^2)}
 ,
\end{align*}
where we applied that derivatives of smooth compactly supported functions, e.g., $\Psi_i\circ\eta_i^{-1}$ and $\chi$,
are bounded. Their norms have been included into the constants appearing in the above inequalities. 
The $H^{-1+s}(\R^2)$-norm of the third summand in~\eqref{eq:appendix_reg_est_RHS_def} can be treated similarly,
i.e., there exists a constant $C$ such that $\|b\|_{H^{-1+s}(\R^2)}\leq C  \|\tilde{a}\|_{C^0(\IS^2)}\|\tilde{u}\|_{H^1(\IS^2)}$.
The second summand in~\eqref{eq:appendix_reg_est_RHS_def} poses no difficulty.
Hence, we conclude that $F\in H^{-1+s}(\R^2)$ and that there exists a constant $C$, which is independent of $\tilde{a}, \tilde{u}$, and $f$, such that
\begin{equation}\label{eq:appendix_reg_est_RHS_est}
 \|F\|_{H^{-1+s}(\R^2)}
 \leq
 C( \|f\|_{H^{-1+s}(\IS^2)}
 +\|\tilde{a}\|_{C^0(\IS^2)}\|\tilde{u}\|_{H^1(\IS^2)}
 +\|\tilde{a}\|_{C^{0,\gamma}(\IS^2)}\|\tilde{u}\|_{H^s(\IS^2)})
 .
\end{equation}
We observe that for every $\xi\in\R^2$, it holds that 
$\xi^\top A\xi \geq \lambda_g \min_{x\in V_i} \sqrt{|g(x)|} \min_{x\in\IS^2}\tilde{a}(x)\xi^\top\xi$ on $\R^2$.
Since the matrix-valued function $A$ is constant on the complement of $V_i$, we observe that there exists a constant $C$ such that
\begin{equation}\label{eq:appendix_reg_est_coeff_matrix_est}
 \sup_{x,y\in\R^2, x\neq y} \frac{\|A(x) - A(y)\|_{\R^{2\times 2}}}{\|x-y\|_{\R^2}^\gamma}
 \leq C \|\tilde{a}\|_{C^{0,\gamma}(\IS^2)}
 .
\end{equation}
We are now in the situation to apply the regularity estimate in~\cite[Lemma~3.2]{ChScTe_SINUM13} to the 
problem in~\eqref{eq:appendix_reg_est_PDE_R2}, which implies that $(u\Psi_i)\circ\eta_i^{-1}\in H^{1+s}(\R^2)$. 
Also it implies together with the estimates in~\eqref{eq:appendix_reg_est_RHS_est} 
and in~\eqref{eq:appendix_reg_est_coeff_matrix_est}
that there exist constants $C_1,C_2,C_3$ such that
\begin{align*}
 \|(\tilde{u}\Psi_i)\circ\eta_i^{-1}\|_{H^{1+s}(\R^2)}
 &\leq C_1 
 \frac{1}{\min_{x\in\IS^2}\tilde{a}(x)} 
 ( \|\tilde{a}\|_{C^{0,\gamma}(\IS^2)}  \|(\tilde{u}\Psi_i)\circ\eta_i^{-1}\|_{H^{1}(\R^2)}
 +\|F\|_{H^{-1+s}(\R^2)})
 \\
 &\quad
 +C_1 \|(\tilde{u}\Psi_i)\circ\eta_i^{-1}\|_{H^{1}(\R^2)}
 \\
 &\leq C_2
 \frac{1}{\min_{x\in\IS^2}\tilde{a}(x)}
 (\|\tilde{a}\|_{C^{0,\gamma}(\IS^2)}  \|\tilde{u}\|_{H^{1}(\IS^2)}
 + \|f\|_{H^{-1+s}(\IS^2)}
 \\
 &\quad 
 +\|\tilde{a}\|_{C^0(\IS^2)}\|\tilde{u}\|_{H^1(\IS^2)}
 +\|\tilde{a}\|_{C^{0,\gamma}(\IS^2)}\|\tilde{u}\|_{H^s(\IS^2)})
 +C_1 \|\tilde{u}\|_{H^1(\IS^2)}
 \\
 &\leq C_3 \Bigl(
 \frac{1}{\min_{x\in\IS^2}\tilde{a}(x)}
 ( \|\tilde{a}\|_{C^{0,\gamma}(\IS^2)}  \|\tilde{u}\|_{H^{1}(\IS^2)}
 + \|f\|_{H^{-1+s}(\IS^2)})
 + \|\tilde{u}\|_{H^1(\IS^2)}\Bigr)
,
\end{align*}
where the first inequality is the estimate from~\cite[Lemma~3.2]{ChScTe_SINUM13} 
applied to our setting.

This argument can be repeated for all remaining $i\in\mathcal{I}$, 
which implies that $\tilde{u}\in H^{1+s}(\IS^2)$, 
and therefore we can establish the previous estimate for every $i\in\mathcal{I}$.
Hence, we sum this squared estimate over all $i\in\mathcal{I}$ 
and take the square root. 
We maximize
the constants over the finite index set 
$\mathcal{I}$ 
which establishes the estimate claimed in the lemma.
\end{proof}

It remains to bound the $H^1(\IS^2)$-norm in the previous lemma with the bound obtained from the Lax--Milgram lemma to obtain the following proposition.

\begin{proposition}\label{prop:reg_est}
For some $0\leq s<\gamma< 1$, let $\tilde{u}\in H^1(\IS^2)/\R$, 
$f\in H^{-1+s}(\IS^2)$, and $\tilde{a}\in C^{0,\gamma}(\IS^2)\cap C^0_+(\IS^2)$ 
 satisfy~\eqref{eq:ellPDEweak},
then, $\tilde{u}\in H^{1+s}(\IS^2)$ and there exists a constant $C$, which 
is independent of $\tilde{u},f$, and $\tilde{a}$, such that
\begin{equation*}
 \|\Phi_{f}(\tilde{a})\|_{H^{1+s}(\IS^2)}
  = \|\tilde{u}\|_{H^{1+s}(\IS^2)} 
 \leq 
 C 
 \|\tilde{a}\|_{C^{0,\gamma}(\IS^2)} \|1/\tilde{a}\|_{C^0(\IS^2)}^2
 \|f\|_{H^{-1+s}(\IS^2)}
 .
\end{equation*}
\end{proposition}

\begin{proof}
 From Lemma~\ref{lemma:reg_est} we readily conclude that $\tilde{u}\in H^{1+s}(\IS^2)$. 
 Also this lemma implies with the $H^1(\IS^2)/\R$-estimate in~\eqref{est:Pth_H1}
 that there exist constants $C_1,C_2$ such that
 \begin{align*}
 \|\tilde{u}\|_{H^{1+s}(\IS^2)} 
 &\leq 
 C_1 
 \Bigl(
 \frac{1}{\min_{x\in\IS^2}\tilde{a}(x)}
 ( \|\tilde{a}\|_{C^{0,\gamma}(\IS^2)}  \|\tilde{u}\|_{H^{1}(\IS^2)}
 + \|f\|_{H^{-1+s}(\IS^2)})
 + \|\tilde{u}\|_{H^1(\IS^2)}\Bigr)
 \\
 &\leq 
 C_2 
  \|\tilde{a}\|_{C^{0,\gamma}(\IS^2)} \|1/\tilde{a}\|_{C^0(\IS^2)}^2
  \|f\|_{H^{-1+s}(\IS^2)}
 ,
\end{align*}
where we applied that $\|\tilde{a}\|_{C^0(\IS^2)}/(\min_{x\in\IS^2}\tilde{a}(x))\geq 1$ and $1/(\min_{x\in\IS^2}\tilde{a}(x))^2 = \|1/\tilde{a}\|_{C^0(\IS^2)}^2$ before summarizing the resulting terms.
\end{proof}

This finishes the proof of the base case in Proposition~\ref{prop:solution_map_regularity}.
In the following we show recursively higher order regularity with the known 
theory for the operator $\Id-\Delta_{\IS^2}$ presented in Section~\ref{sec:iGRF}
to analyze the domain and the respective range of $\Phi_{f}$ more precisely.

\begin{proposition}
Let $\iota\in\N_0$, $\gg\in(0,1)$, and $s\in[0,+\infty)$
satisfy $s<\iota+\gg$. If $f\in H^{-1+s}(\IS^2)$, 
then it holds that
 \begin{equation*}
 \Phi_{f} : 
 C^{\iota,\gg}(\IS^2)\cap C^0_+(\IS^2)
 \rightarrow 
 H^{1+s}(\IS^2)
 \end{equation*}
 is continuous with respect to the topology of $C^{\iota,\gg}(\IS^2)$.
 
 Moreover if $s\geq 1$, then for every $n\in\{0,\dots,\lfloor s \rfloor-1\}$, there exists a constant $C>0$ 
 such that for every $\tilde{a}\in C^{\iota,\gg}(\IS^2)\cap C^0_+(\IS^2)$,
 \begin{align*}
 &\|\Phi_{f}(\tilde{a})\|_{H^{1+(n+1)+\{s\}}(\IS^2)}\\
 &\qquad \leq C\|1/\tilde{a}\|_{C^{n,\gg}(\IS^2)} \bigl( \|f\|_{H^{1+(n-1)+\{s\}}(\IS^2)}
 + \|\tilde{a}\|_{C^{n+1,\gg}(\IS^2)}\|\Phi_{f}(\tilde{a})\|_{H^{1+n+\{s\}}(\IS^2)}\bigr)
\end{align*}
 where $\{s\}$ denotes the fractional part of $s$.
\end{proposition}

\begin{proof}
The case $s\in[0,1)$ will serve as a base case for an induction argument.
There the case $s=0$ is already known from Proposition~\ref{prop:solution_map_cont}.
So let $s\in(0,1)$ and assume that $\iota=0$ and $\gamma\in(s,1)$.
From Proposition~\ref{prop:reg_est} we infer that $\Phi_{f}(\tilde{a})\in H^{1+s}(\IS^2)$,
which establishes the claimed domain and range of $\Phi_f$.
To prove the continuity of $\Phi_{f}$ let 
$(\tilde{a}_j,j\in\N_0)$ be a sequence in $ C^{0,\gg}(\IS^2)\cap\ C^0_+(\IS^2)$
such that $\|\tilde{a}_j - \tilde{a}_0\|_{ C^{0,\gg}(\IS^2)}\rightarrow 0$ as $j\to+\infty$.
We observe that for every $j\in\N$, it holds that
\begin{equation}\label{eq:sol_difference_H-1}
 (\tilde{a}_0\nabla_{\IS^2}(\Phi_{f}(\tilde{a}_0)-\Phi_{f}(\tilde{a}_j)),\nabla_{\IS^2} v)
 =
 (-(\tilde{a}_0-\tilde{a}_j)\nabla_{\IS^2}\Phi_{f}(\tilde{a}_j),\nabla_{\IS^2} v)
 \quad \forall v\in H^1(\IS^2)/\R
 .
\end{equation}
Since $\Phi_{f}(\tilde{a}_j)\in H^{1+s}(\IS^2)$, $j\in\N$, 
we obtain with Proposition~\ref{prop:Hoelder_Sobolev_prod}
that there exist constants $C_1,C_2$ such that
\begin{align*}
\|\nabla_{\IS^2}\cdot((\tilde{a}_0-\tilde{a}_j)\nabla_{\IS^2}\Phi_{f}(\tilde{a}_j))\|_{H^{-1+s}(\IS^2)}
&\leq C_1 
\|(\tilde{a}_0-\tilde{a}_j)\nabla_{\IS^2}\Phi_{f}(\tilde{a}_j)\|_{H^{s}(\IS^2)}
\\
&\leq C_2
\|\tilde{a}_0-\tilde{a}_j\|_{C^{0,\gamma}(\IS^2)} \|\Phi_{f}(\tilde{a}_j)\|_{H^{1+s}(\IS^2)}
.
\end{align*}
Hence, Proposition~\ref{prop:reg_est} applied to the setting in~\eqref{eq:sol_difference_H-1}
implies that there exists a constant $C$, which is independent of $(\tilde{a}_j,j\in\N_0)$ and $f$, 
such that for every $j\in\N$, it holds that
\begin{equation*}
 \|\Phi_{f}(\tilde{a}_0) - \Phi_{f}(\tilde{a}_j)\|_{H^{1+s}(\IS^2)}
 \leq C 
 \frac{\|\tilde{a}_0\|_{C^{0,\gg}(\IS^2)}}{(\min_{x\in\IS^2}\tilde{a}_0(x))^2}
 \frac{\|\tilde{a}_j\|_{C^{0,\gg}(\IS^2)}}{(\min_{x\in\IS^2}\tilde{a}_j(x))^2}
 \|f\|_{H^{-1+s}(\IS^2)} \|\tilde{a}_0- \tilde{a}_j\|_{C^{0,\gg}(\IS^2)} 
 .
\end{equation*}
We have $\|\tilde{a}_0-\tilde{a}_j\|_{C^0(\IS^2)}=:\epsilon_j\rightarrow 0$ as $j\rightarrow +\infty$, 
which implies that $\epsilon_j\leq 1/2 \min_{x\in\IS^2}\tilde{a}_0(x)$ for 
every $j$ that are sufficiently large, i.e., $j > j_0$ for some $j_0 \in \N$.
Since $\tilde{a}_j(x')\geq \min_{x\in\IS^2}\tilde{a}_0(x) - \epsilon_j$ for 
every $x'\in\IS^2$, 
we obtain that $1/ \min_{x\in\IS^2}\tilde{a}_j(x)\leq 2/\min_{x\in\IS^2}\tilde{a}_{0}(x)$ 
for every $j>j_0$.
Since $\|\tilde{a}_j\|_{C^{0,\gg}(\IS^2)}$ and $(\min_{x\in\IS^2}\tilde{a}_j(x))^{-2}$ 
can be bounded independently of $j$,
it follows that $\|\Phi_{f}(\tilde{a}_j) - \Phi_{f}(\tilde{a}_0)\|_{H^{1+s}(\IS^2)} \rightarrow 0$ 
as $j\to +\infty$, i.e.,
$\Phi_f:C^{0,\gg}(\IS^2)\cap C^0_+(\IS^2)\rightarrow H^{1+s}(\IS^2)$
is continuous.

For $s\geq 1$, it must hold that $\iota\geq1$. 
Since $\tilde{a}\in C^{\iota,\gg}(\IS^2)\cap C^0_+(\IS^2)$ and $\tilde{u}:=\Phi_{f}(\tilde{a})$,
for any $w\in H^1(\IS^2)$, we take 
$w/\tilde{a} - 1/|\IS^2|\int_{\IS^2} w/\tilde{a}\mathrm d\sigma \in H^1(\IS^2)/\R$ 
as a test function and thus rewrite the PDE in~\eqref{eq:ellPDEweak} as
\begin{equation*}
 (\tilde{u} , w ) + \Bigl(f,\frac{w}{\tilde{a}}\Bigr)
 =
 (\tilde{u} , w ) + \Bigl(\tilde{a}\nabla_{\IS^2}\tilde{u},\nabla_{\IS^2}\Bigl(\frac{w}{\tilde{a}}\Bigr)\Bigr)
 =
 (\tilde{u} , w ) + (\nabla_{\IS^2}\tilde{u},\nabla_{\IS^2}w) 
 - \Bigl(\frac{\nabla_{\IS^2}\tilde{a}\cdot\nabla_{\IS^2}\tilde{u}}{\tilde{a}},w\Bigr)
 ,
\end{equation*}
where we applied that $(f,1)=0$.
Hence, for every $w\in H^1(\IS^2)$, it holds that
\begin{equation*}
 (\tilde{u} , w ) + (\nabla_{\IS^2}\tilde{u},\nabla_{\IS^2}w)
 =
 \Bigl(\frac{f + \nabla_{\IS^2}\tilde{a}\cdot\nabla_{\IS^2}\tilde{u}}{\tilde{a}},w\Bigr) + (\tilde{u} , w )
 ,
\end{equation*}
which is stated with equality in $H^{-1}(\IS^2)$ as
\begin{equation}\label{eq:Reg_H^1+s}
 (\Id - \Delta_{\IS^2}) \tilde{u} 
 =
 \frac{f + \nabla_{\IS^2}\tilde{a}\cdot\nabla_{\IS^2}\tilde{u}}{\tilde{a}} + \tilde{u}
 =:
 F.
\end{equation}
We observe with~\eqref{eq:inv_Id-Beltrami} 
that $(\Id - \Delta_{\IS^2})^{-1}$ is a linear and bounded operator from 
$H^{r}(\IS^2)$ to $H^{r+2}(\IS^2)$ for every $r\in\R$.
The claim is now shown by induction.
Let us write $s=\lfloor s\rfloor + \{s\}$, 
where $\{s\}\in[0,1)$ is the fractional part of $s$, and assume as induction hypothesis that 
$\Phi_f:C^{n,\gg}(\IS^2)\cap C^0_+(\IS^2)\rightarrow H^{1+n+\{s\}}(\IS^2)$
is continuous for every  $n\in\{0,1,\dots,\lfloor s\rfloor -1\}$,
which we already showed for $n=0$.
Let $n\in\{0,1,\dots,\lfloor s\rfloor -1\}$ and let 
$\tilde{a}\in C^{n+1,\gg}(\IS^2)\cap C^0_+(\IS^2)$.
Since by our induction hypothesis $\tilde{u}=\Phi_{f}(\tilde{a})\in H^{1+n+\{s\}}(\IS^2)$,
we conclude with Proposition~\ref{prop:Hoelder_Sobolev_prod}
that the right hand side $F$ in~\eqref{eq:Reg_H^1+s} is in 
$H^{1+(n-1)+\{s\}}$. 
The fact that $(\Id - \Delta_{\IS^2})^{-1}$ is a linear and bounded 
operator from $H^{1+(n-1)+\{s\}}(\IS^2)$ to $H^{1+(n+1)+\{s\}}(\IS^2)$
implies that 
$\tilde{u}=\Phi_{f}(\tilde{a})\in H^{1+(n+1)+\{s\}}(\IS^2)$.
Moreover it implies with Proposition~\ref{prop:Hoelder_Sobolev_prod} a regularity estimate for $\tilde{u}=\Phi_{f}(\tilde{a})$, 
i.e., there exist constants $C_1,C_2,C_3$ that are independent of $\tilde{a}$ and $f$ such that
\begin{align*}
 \|\Phi_{f}(\tilde{a})&\|_{H^{1+(n+1)+\{s\}}(\IS^2)}\\
 &\leq C_1\|F\|_{H^{1+(n-1)+\{s\}}(\IS^2)}\\
 &\leq C_2\Bigl(\|1/\tilde{a}\|_{C^{n,\gg}(\IS^2)} \bigl( \|f\|_{H^{1+(n-1)+\{s\}}(\IS^2)}
 + \|\tilde{a}\|_{C^{n+1,\gg}(\IS^2)}\|\Phi_{f}(\tilde{a})\|_{H^{1+n+\{s\}}(\IS^2)}\bigr)\\
 & \hspace*{3em} + \|\Phi_{f}(\tilde{a})\|_{H^{1+(n-1)+\{s\}}(\IS^2)}\Bigr)\\
 & \leq C_2 \|1/\tilde{a}\|_{C^{n,\gg}(\IS^2)} \bigl( \|f\|_{H^{1+(n-1)+\{s\}}(\IS^2)}
    + \|\tilde{a}\|_{C^{n+1,\gg}(\IS^2)}\|\Phi_{f}(\tilde{a})\|_{H^{1+n+\{s\}}(\IS^2)}\bigr),
\end{align*}
where the last inequality holds since 
  \begin{equation*}
    \|\Phi_{f}(\tilde{a})\|_{H^{1+(n-1)+\{s\}}(\IS^2)} 
      \leq \|\Phi_{f}(\tilde{a})\|_{H^{1+n+\{s\}}(\IS^2)}
  \end{equation*}
and $\|1/\tilde{a}\|_{C^{n,\gg}(\IS^2)}\|\tilde{a}\|_{C^{n+1,\gg}(\IS^2)} \geq 1$.
This is the desired recursion formula and implies the claimed domain and range of $\Phi_{f}$.
To prove continuity of $\Phi_{f}$ let 
$(\tilde{a}_j,j\in\N_0)$ be a sequence in $ C^{n+1,\gg}(\IS^2)\cap C^0_+(\IS^2)$
such that $\|\tilde{a}_j - \tilde{a}_0\|_{ C^{n+1,\gg}(\IS^2)}\rightarrow 0$ as $j\to+\infty$
and let $(\tilde{u}_j=\Phi_{f}(\tilde{a}_j),j\in\N_0)$ be the sequence of respective solutions.
The same manipulations that showed~\eqref{eq:Reg_H^1+s} imply with~\eqref{eq:sol_difference_H-1} that
\begin{equation*}
 (\Id-\Delta_{\IS^2})(\tilde{u}_0 - \tilde{u}_j) 
 =\frac{\nabla_{\IS^2}\cdot((\tilde{a}_0-\tilde{a}_j)\nabla_{\IS^2}\tilde{u}_j) + \nabla_{\IS^2}\tilde{a}_0\cdot\nabla_{\IS^2}(\tilde{u}_0 - \tilde{u}_j)}
	{\tilde{a}_0} 
	+ (\tilde{u}_0 - \tilde{u}_j).
\end{equation*}
Similar estimates as for $\Phi_{f}(\tilde{a})$ above imply that
\begin{align*}
 \|\Phi_{f}(\tilde{a}_0) - \Phi_{f}(\tilde{a}_j)&\|_{H^{1+(n+1)+\{s\}}(\IS^2)}\\
 &\leq C'\Bigl( \|1/\tilde{a}_0\|_{C^{n,\gg}(\IS^2)}
	\bigl( \|\tilde{a}_0-\tilde{a}_j\|_{C^{n+1,\gg}(\IS^2)}\|\Phi_{f}(\tilde{a}_0)\|_{H^{1+n+\{s\}}(\IS^2)}\\
	&\hspace*{10em}	+ \|\tilde{a}_0\|_{C^{n+1,\gg}(\IS^2)}\|\Phi_{f}(\tilde{a}_0) - \Phi_{f}(\tilde{a}_j)\|_{H^{1+n+\{s\}}(\IS^2)} \bigr)\\
	&\hspace*{3em} + \|\Phi_{f}(\tilde{a}_0) - \Phi_{f}(\tilde{a}_j)\|_{H^{1+(n-1)+\{s\}}(\IS^2)}\Bigr).
\end{align*}
Since by our induction hypothesis 
$\Phi_f:C^{n,\gg}(\IS^2)\cap C^0_+(\IS^2)\rightarrow H^{1+n+\{s\}}(\IS^2)$
is continuous, 
$\|\Phi_{f}(\tilde{a}_0) - \Phi_{f}(\tilde{a}_j)\|_{H^{n+\{s\}}(\IS^2)}\leq\|\Phi_{f}(\tilde{a}_0) - \Phi_{f}(\tilde{a}_j)\|_{H^{1+n+\{s\}}(\IS^2)}\rightarrow 0$ as $j\to+\infty$,
which implies with  $\|\tilde{a}_0 - \tilde{a}_j\|_{ C^{n+1,\gg}(\IS^2)}\rightarrow 0$ as $j\to+\infty$
that $\|\Phi_{f}(\tilde{a}_0) - \Phi_{f}(\tilde{a}_j)\|_{H^{1+(n+1)+\{s\}}(\IS^2)}\rightarrow 0$ as $j\to+\infty$,
i.e., 
$\Phi_f:C^{n+1,\gg}(\IS^2)\cap C^0_+(\IS^2)\rightarrow H^{1+(n+1)+\{s\}}(\IS^2)$
is continuous.
This finishes the induction and the proof of the proposition.
\end{proof}

We finish this appendix by remarking that the $H^{1+s}(\IS^2)$-regularity for every  $s<\gb/2$ 
of the solution can also be deduced from
higher order H\"older regularity, which is implied by
Schauder estimates, cp.\ ~\cite[Chapters~6 and~8]{GilTrud2ndEd}, applied to pullbacks 
of the solution to the chart domains. 
Specifically, the continuous embedding $C^{\iota,\gg}(\IS^2)\subset H^{s'}(\IS^2)$ for 
$\iota\in\N_0, \gg\in(0,1)$, and $s'\geq0$ such that $\iota+\gg> s'$,
which is an immediate consequence of Proposition~\ref{prop:Hoelder_Sobolev_prod},
would imply $H^{1+s}(\IS^2)$-regularity.
Since the explicit dependence of the coefficients of the elliptic operator
in these estimates is analyzed in \cite[Section~8.2]{Herrmann13},
$L^p$-integrability could also be deduced.

\section{Finite Element convergence for elliptic PDEs}\label{app:proof:prop:solution_map_approx_FE}

\begin{proof}[Proof of Proposition~\ref{prop:solution_map_approx_FE}]
We observe that Proposition~\ref{prop:solution_map_regularity} implies
that $\Phi_{f}(\tilde{a})\in H^{1+s}(\IS^2)$.
The approximation property 
for integer orders, cp.\ ~\cite[Proposition~2.7]{Demlow2009}, 
implies by interpolation
that for every $h$ and $k$, there exists an interpolation operator 
$I^{h,k}$ which is, for every $s>0$, continuous from 
$H^{1+s}(\IS^2) \rightarrow S^{k}(\IS^2,\cT_h)$
and a constant $C_s > 0$ such that for every $h>0$ and for every 
function $v\in H^{1+s}(\IS^2)$, it holds that
\begin{equation}\label{eq:FEIntErr}
\| v - I^{k,h} v \|_{H^1(\IS^2)} 
\leq 
C_s \;  h^{\min\{s,k\}} \; \| v \|_{H^{1+s}(\IS^2)} 
,
\end{equation}
where $C_s>0$ is independent of $h$ 
but depends on $s$.
The coercivity and Galerkin orthogonality imply
in the usual fashion that for every $v^h\in V^{h,k}$, it holds that
\begin{align}\label{eq:Ceas_lemma}
\| \Phi_{f}(\tilde{a}) - \Phi_{f}^{h,k}(\tilde{a}) \|^2_{H^1(\IS^2)/\R} 
& \leq
\frac{1}{\min_{x\in\IS^2}\tilde{a}(x)}
(\tilde{a} \nabla_{\IS^2} ( \Phi_{f}(\tilde{a}) - \Phi_{f}^{h,k}(\tilde{a})) ,  
                 \nabla_{\IS^2}  (\Phi_{f}(\tilde{a}) - \Phi_{f}^{h,k}(\tilde{a})) )
\nonumber\\
& = 
\frac{1}{\min_{x\in\IS^2}\tilde{a}(x)}
(\tilde{a} \nabla_{\IS^2} ( \Phi_{f}(\tilde{a}) - \Phi_{f}^{h,k}(\tilde{a})) ,  \nabla_{\IS^2}  (\Phi_{f}(\tilde{a}) - v^{h}) )
\\
& \leq 
\frac{\|\tilde{a}\|_{C^0(\IS^2)}}{\min_{x\in\IS^2}\tilde{a}(x)}
\|  \Phi_{f}(\tilde{a}) - \Phi_{f}^{h,k}(\tilde{a}) \|_{H^1(\IS^2)/\R}
\|  \Phi_{f}(\tilde{a}) - v^h \|_{H^1(\IS^2)/\R}
.\nonumber
\end{align}
When we equip $H^1(\IS^2)/\R$ with the $H^1(\IS^2)$-norm, 
the orthogonal decomposition $H^1(\IS^2) = H^1(\IS^2)/\R \oplus \Span\{1\}$ holds,
which implies with~\eqref{eq:Ceas_lemma} that
\begin{align*}
\|  \Phi_{f}(\tilde{a}) - \Phi_{f}^{h,k}(\tilde{a}) \|_{H^1(\IS^2)/\R}
&\leq 
\frac{\|\tilde{a}\|_{C^0(\IS^2)}}{\min_{x\in\IS^2}\tilde{a}(x)} 
\inf_{v^h \in V^{h,k}} 
\| \Phi_{f}(\tilde{a}) - v^h\|_{H^1(\IS^2)/\R}
\\
&\leq
\frac{\|\tilde{a}\|_{C^0(\IS^2)}}{\min_{x\in\IS^2}\tilde{a}(x)} 
\inf_{v^h \in V^{h,k}} 
\| \Phi_{f}(\tilde{a}) - v^h\|_{H^1(\IS^2)}
\\
&=
\frac{\|\tilde{a}\|_{C^0(\IS^2)}}{\min_{x\in\IS^2}\tilde{a}(x)} 
\inf_{v^h \in S^{k}(\IS^2,\cT_h)} 
\| \Phi_{f}(\tilde{a}) - v^h\|_{H^1(\IS^2)}
,
\end{align*}
where we also used that $\Phi_{f}(\tilde{a})\in H^1(\IS^2)/\R$.
Now the claim follows with~\eqref{eq:FEIntErr}.
\end{proof}
\bibliographystyle{plain}
\bibliography{diffs2_revise1}

\begin{thebibliography}{10}

\bibitem{BL12_2}
Andrea Barth and Annika Lang.
\newblock Multilevel {M}onte {C}arlo method with applications to stochastic
  partial differential equations.
\newblock {\em Int. J. Comput. Math.}, 89(18):2479--2498, 2012.

\bibitem{BSZ11}
Andrea Barth, Christoph Schwab, and N.~Zollinger.
\newblock Multi-level {M}onte {C}arlo finite element method for elliptic {PDE}s
  with stochastic coefficients.
\newblock {\em Numerische Mathematik}, 119(1):123--161, 2011.

\bibitem{Bogachev07}
Vladimir~Igorevich Bogachev.
\newblock {\em Measure Theory. {V}ol. {I}, {II}}.
\newblock Springer-Verlag, Berlin, 2007.

\bibitem{Bramble_Multigrid}
James~H. Bramble.
\newblock {\em Multigrid Methods}.
\newblock Pitman Research Notes in Mathematics Series, Vol. 294. Longman
  Scientific \& Technical, Harlow; copublished in the United States with John
  Wiley \& Sons, Inc., New York, 1993.

\bibitem{Charrier_SINUM12}
Julia Charrier.
\newblock Strong and weak error estimates for elliptic partial differential
  equations with random coefficients.
\newblock {\em SIAM J. Numer. Anal.}, 50(1):216--246, 2012.

\bibitem{ChScTe_SINUM13}
Julia Charrier, Robert Scheichl, and Aretha Teckentrup.
\newblock Finite element error analysis for elliptic {PDE}s with random
  coefficients and applications.
\newblock {\em SIAM J. Numer. Anal.}, 50(1):322--352, 2013.

\bibitem{DPZ2014}
Giuseppe {Da Prato} and Jerzy Zabczyk.
\newblock {\em Stochastic Equations in Infinite Dimensions}.
\newblock Encyclopedia of Mathematics and Its Applications, Vol. 152.
  Cambridge: Cambridge University Press, second edition, 2014.

\bibitem{Demlow2009}
Alan Demlow.
\newblock Higher-order finite element methods and pointwise error estimates for
  elliptic problems on surfaces.
\newblock {\em SIAM Journal on Numerical Analysis}, 47(2):805--827, 2009.

\bibitem{dunford1958LinOpeIGenThe}
Nelson Dunford and Jacob~T. Schwartz.
\newblock {\em Linear {O}perators. {I}. {G}eneral {T}heory}.
\newblock With the assistance of W. G. Bade and R. G. Bartle. Pure and Applied
  Mathematics, Vol. 7. Interscience Publishers, Inc., New York; Interscience
  Publishers, Ltd., London, 1958.

\bibitem{Dziuk88}
Gerhard Dziuk.
\newblock Finite elements for the {B}eltrami operator on arbitrary surfaces.
\newblock In {\em Partial Differential Equations and Calculus of Variations},
  Lecture Notes in Mathematics, Vol. 1357, pages 142--155. Springer-Verlag,
  Berlin, 1988.

\bibitem{gMLQMC}
Robert~N. Gantner.
\newblock {A Generic C++ Library for Multilevel Quasi-Monte Carlo}.
\newblock In {\em Proceedings of the Platform for Advanced Scientific Computing
  Conference}, PASC '16, pages 11:1--11:12, New York, NY, USA, 2016. ACM.

\bibitem{GilTrud2ndEd}
David Gilbarg and Neil~S. Trudinger.
\newblock {\em Elliptic Partial Differential Equations of Second Order}.
\newblock Grundlehren der Mathematischen Wissenschaften [Fundamental Principles
  of Mathematical Sciences], Vol. 224. Springer-Verlag, Berlin, second edition,
  1983.

\bibitem{GilesMLMC2015}
Michael~B. Giles.
\newblock Multilevel {M}onte {C}arlo methods.
\newblock {\em Acta Numer.}, 24:259--328, 2015.

\bibitem{GKSS_MLMC13}
Claude~J. Gittelson, Juho K\"onn\"o, Christoph Schwab, and Rolf Stenberg.
\newblock The multi-level {M}onte {C}arlo finite element method for a
  stochastic {B}rinkman problem.
\newblock {\em Numerische Mathematik}, 125(2):347, 2013.

\bibitem{gknsss15}
Ivan~G. Graham, Frances~Y. Kuo, James~A. Nichols, Robert Scheichl, Christoph
  Schwab, and Ian~H. Sloan.
\newblock Quasi-{M}onte {C}arlo finite element methods for elliptic {PDE}s with
  lognormal random coefficient.
\newblock {\em Numerische Mathematik}, 131(2):329--368, 2015.

\bibitem{GS13}
Nadine Grosse and Cornelia Schneider.
\newblock Sobolev spaces on {R}iemannian manifolds with bounded geometry:
  general coordinates and traces.
\newblock {\em Math.Nach.}, 286:1586--1613, 2013.

\bibitem{hackbusch2010Elldifequ}
Wolfgang Hackbusch.
\newblock {\em Elliptic Differential Equations}.
\newblock Springer Series in Computational Mathematics, Vol. 18.
  Springer-Verlag, Berlin, {E}nglish edition, 2010.
\newblock Theory and numerical treatment. Translated from the 1986 corrected
  German edition by Regine Fadiman and Patrick D. F. Ion.

\bibitem{HRKM03}
Dennis M.~jun. Healy, Daniel~N. Rockmore, Peter~J. Kostelec, and Sean Moore.
\newblock {FFT}s for the 2-sphere--improvements and variations.
\newblock {\em J. Fourier Anal. Appl.}, 9(4):341--385, 2003.

\bibitem{HeinrichMLMC}
Stefan Heinrich.
\newblock Multilevel {M}onte {C}arlo methods.
\newblock In {\em Large-Scale Scientific Computing}, Lecture Notes in Computer
  Science, Vol. 2179, pages 58--67. Springer-Verlag, Berlin, 2001.

\bibitem{Herrmann13}
Lukas Herrmann.
\newblock Isotropic random fields on the sphere - stochastic heat equation and
  regularity of random elliptic {PDE}s.
\newblock Master's thesis, ETH Z\"urich, October 2013.

\bibitem{hille1974Funanasem}
Einar Hille and Ralph~S. Phillips.
\newblock {\em Functional Analysis and Semi-Groups}.
\newblock American Mathematical Society, Providence, R.I., 1974.
\newblock Third printing of the revised edition of 1957, American Mathematical
  Society Colloquium Publications, Vol. XXXI.

\bibitem{BETL}
Ralf Hiptmair and Lars Kielhorn.
\newblock {BETL} -- a generic boundary element template library.
\newblock Technical Report 2012-36, Seminar for Applied Mathematics, ETH
  Z{\"u}rich, Switzerland, 2012.

\bibitem{Hoermander1976}
Lars H\"ormander.
\newblock The boundary problems of physical geodesy.
\newblock {\em Archive for Rational Mechanics and Analysis}, 62(1):1--52, 1976.

\bibitem{matplotlib}
John~D. Hunter.
\newblock Matplotlib: A 2d graphics environment.
\newblock {\em Computing In Science \& Engineering}, 9(3):90--95, 2007.

\bibitem{jost2011RieGeoGeoAna6ed}
J\"urgen Jost.
\newblock {\em Riemannian Geometry and Geometric Analysis}.
\newblock Universitext. Springer-Verlag, Berlin, sixth edition, 2011.

\bibitem{L16}
Annika Lang.
\newblock A note on the importance of weak convergence rates for {SPDE}
  approximations in multilevel {M}onte {C}arlo schemes.
\newblock In Ronald Cools and Dirk Nuyens, editors, {\em Monte Carlo and
  Quasi-Monte Carlo Methods, MCQMC, Leuven, Belgium, April 2014}, volume 163 of
  {\em Springer Proceedings in Mathematics \& Statistics}, pages 489--505,
  2016.

\bibitem{LS15}
Annika Lang and Christoph Schwab.
\newblock Isotropic {G}aussian random fields on the sphere: Regularity, fast
  simulation and stochastic partial differential equations.
\newblock {\em Ann. Appl. Probab.}, 25(6):3047--3094, 12 2015.

\bibitem{loeve1977ProtheI}
Michel Lo\`eve.
\newblock {\em Probability Theory {I}}.
\newblock Springer-Verlag, New York, fourth edition, 1977.
\newblock Graduate Texts in Mathematics, Vol. 45.

\bibitem{MP11}
Domenico Marinucci and Giovanni Peccati.
\newblock {\em Random Fields on the Sphere. Representation, Limit Theorems and
  Cosmological Applications}.
\newblock Cambridge: Cambridge University Press, 2011.

\bibitem{M98}
Mitsuo Morimoto.
\newblock {\em Analytic Functionals on the Sphere}.
\newblock Translations of Mathematical Monographs, Vol. 178. Providence, RI:
  American Mathematical Society, 1998.

\bibitem{Ned2001}
Jean-Claude N\'ed\'elec.
\newblock {\em Acoustic and Electromagnetic Equations}.
\newblock Applied Mathematical Sciences, Vol. 144. Springer-Verlag, New York,
  2001.

\bibitem{rudin1973FA}
Walter Rudin.
\newblock {\em Functional Analysis}.
\newblock McGraw-Hill Series in Higher Mathematics. McGraw-Hill Book Co., New
  York-D\"{u}sseldorf-Johannesburg, 1973.

\bibitem{SaScBEM11}
Stefan~A. Sauter and Christoph Schwab.
\newblock {\em Boundary Element Methods}.
\newblock Springer Series in Computational Mathematics, Vol. 39.
  Springer-Verlag, Berlin, 2011.
\newblock Translated and expanded from the 2004 German original.

\bibitem{SHTns2013}
Nathana{\"e}l Schaeffer.
\newblock Efficient spherical harmonic transforms aimed at pseudospectral
  numerical simulations.
\newblock {\em Geochemistry, Geophysics, Geosystems}, 14(3):751--758, 2013.

\bibitem{Schenk200169}
Olaf Schenk, Klaus G{\"a}rtner, Wolfgang Fichtner, and Andreas Stricker.
\newblock Pardiso: a high-performance serial and parallel sparse linear solver
  in semiconductor device simulation.
\newblock {\em Future Generation Computer Systems}, 18(1):69 -- 78, 2001.
\newblock I. High Performance Numerical Methods and Applications. II.
  Performance Data Mining: Automated Diagnosis, Adaption, and Optimization.

\bibitem{strichartz1983AnaLapcomRieman}
Robert~S. Strichartz.
\newblock Analysis of the {L}aplacian on the complete {R}iemannian manifold.
\newblock {\em Journal of Functional Analysis}, 52(1):48--79, 1983.

\bibitem{Szego}
G\'abor Szeg\H{o}.
\newblock {\em Orthogonal Polynomials}.
\newblock American Mathematical Society, Providence, R.I., fourth edition,
  1975.
\newblock American Mathematical Society, Colloquium Publications, Vol. XXIII.

\bibitem{taylor1981PDO}
Michael~E. Taylor.
\newblock {\em Pseudodifferential Operators}.
\newblock Princeton Mathematical Series, Vol. 34. Princeton University Press,
  Princeton, N.J., 1981.

\bibitem{T83}
Hans Triebel.
\newblock {\em Theory of Function Spaces}.
\newblock Monographs in Mathematics, Vol. 78. Basel-Boston-Stuttgart:
  Birkh\"{a}user Verlag, 1983.

\bibitem{Triebel92}
Hans Triebel.
\newblock {\em Theory of Function Spaces {II}}.
\newblock Monographs in Mathematics, Vol. 84. Birkh\"{a}user Verlag, Basel,
  1992.

\bibitem{Triebel95}
Hans Triebel.
\newblock {\em Interpolation Theory, Function Spaces, Differential Operators}.
\newblock Johann Ambrosius Barth, Heidelberg, second edition, 1995.

\end{thebibliography}
\end{document}